\documentclass[11pt]{amsart}

\usepackage{latexsym}
\usepackage{amsmath}
\usepackage{amsthm}
\usepackage{amssymb}
\usepackage{amsfonts}
\usepackage{comment}

\usepackage{mathrsfs}
\usepackage[ocgcolorlinks, linkcolor=blue]{hyperref}

\usepackage{calc}
             {\begin{list}{\arabic{enumi}.}{\usecounter{enumi}%
              \setlength{\labelsep}{0.5em}%
              \settowidth{\labelwidth}{\arabic{enumi}.}%
              \setlength{\leftmargin}{\labelwidth+\labelsep}}}%
             {\end{list}}

             {\begin{list}{$\bullet$}{\usecounter{enumi}%
              \setlength{\labelsep}{0.5em}%
              \settowidth{\labelwidth}{*}%
              \setlength{\leftmargin}{\labelwidth+\labelsep}}}%
             {\end{list}}

\newcommand{\mR}{\mathbb{R}}                    % Formatting for R
\newcommand{\mC}{\mathbb{C}}                    % Formatting for C
\newcommand{\mN}{\mathbb{N}}                    % Formatting for N
\newcommand{\abs}[1]{\lvert #1 \rvert}          % Formatting for the absolute value
\newcommand{\norm}[1]{\lVert #1 \rVert}         % Formatting for the norm
\newcommand{\ol}[1]{\overline{#1}}

\newcommand{\mD}{\mathbb{D}}

\newcommand{\ehat}{\,\hat{\rule{0pt}{6pt}}\,}

\newcommand{\etilde}{\,\tilde{\rule{0pt}{6pt}}\,}

\newcommand{\im}{\mathrm{Im}}

\newcommand{\mF}{\mathscr{F}}

\newcommand{\p}{\partial}
\newcommand{\closure}[1]{\overline{#1}}

\newcommand{\mdiv}{\mathrm{div}}

\newcommand{\mA}{\mathcal{A}}

\newcommand{\eps}{\varepsilon}

\newtheorem{theorem}{Theorem}[section]

\newtheorem{lemma}[theorem]{Lemma}

\newtheorem{proposition}[theorem]{Proposition}
\newtheorem{corollary}[theorem]{Corollary}

\theoremstyle{definition}
\newtheorem*{definition}{Definition}
\newtheorem{example}[theorem]{Example}

\newtheorem{remark}[theorem]{Remark}

\newtheorem*{notation}{Notation}
\newtheorem*{motivation}{Motivation}

\numberwithin{equation}{section}

\begin{document}

\title{Applications of microlocal analysis to inverse problems}
\author{Mikko Salo}
\address{Department of Mathematics and Statistics, University of Jyv\"askyl\"a}
\email{mikko.j.salo@jyu.fi}

\begin{abstract}
These are lecture notes for a minicourse on applications of microlocal analysis in inverse problems, to be given in Helsinki and Shanghai in June 2019.
\end{abstract}

\maketitle

%{\begin{center} \huge Applications of microlocal analysis to inverse problems \\[15pt]
%{\Large Lecture notes, Summer 2019} \\[15pt] {\large Mikko Salo \\[3pt] Department of Mathematics and Statistics \\[-7pt]ÊUniversity of Jyv\"askyl\"a} 
%\end{center}}

%\author{Mikko Salo \\[30pt] {\normalsize\scshape Department of Mathematics and Statistics} \\[-12pt] {\normalsize\scshape University of Jyv\"askyl\"a}}

%\frontmatter

%\address{Department of Mathematics and Statistics, University of Helsinki}
%\email{mikko.salo@helsinki.fi}
%\date{\today}

%\setcounter{page}{4}
%\tableofcontents

%\mainmatter

%\chapter*{Preface}

\section*{Preface}

Microlocal analysis originated in the 1950s, and by now it is a substantial mathematical theory with many different facets and applications. One might view microlocal analysis as 
\begin{itemize}
\item
a kind of "variable coefficient Fourier analysis" for solving variable coefficient PDEs; or 
\item 
as a theory of \emph{pseudodifferential operators} ($\Psi$DOs) and \emph{Fourier integral operators} (FIOs); or 
\item 
as a phase space (or time-frequency) approach to studying functions, operators and their singularities (\emph{wave front sets}).
\end{itemize}
$\Psi$DOs were introduced by Kohn and Nirenberg \cite{KohnNirenberg}, and FIOs and wave front sets were studied systematically by H\"ormander \cite{Hormander_FIO}. Much of the theory up to the early 1980s is summarized in the four volume treatise of H\"ormander \cite{Hormander}. There are remarkable applications of microlocal analysis and related ideas in many fields of mathematics. Classical examples include spectral theory and the Atiyah-Singer index theorem, and more recent examples include scattering theory, behavior of chaotic systems, inverse problems, and general relativity.

In this minicourse we will try to describe some classical applications of microlocal analysis to inverse problems, together with a very rough non-technical overview of relevant parts of microlocal analysis. In a nutshell, here are a few typical applications:
\begin{enumerate}
\item[1.] 
{\bf Computed tomography / X-ray transform:} the X-ray transform is an FIO, and under certain conditions its normal operator is an elliptic $\Psi$DO. Microlocal analysis can be used to predict which sharp features (singularities) of the image can be reconstructed in a stable way from limited data measurements. Microlocal analysis is also a powerful tool in the study of geodesic X-ray transforms related to seismic imaging applications.
\item[2.] 
{\bf Calder\'on problem / Electrical Impedance Tomography:} the boundary measurement map (Dirichlet-to-Neumann map) is a $\Psi$DO, and the boundary values of the conductivity as well as its derivatives can be computed from the symbol of this $\Psi$DO.
\item[3.] 
{\bf Gel'fand problem / seismic imaging:} the boundary measurement operator (hyperbolic Dirichlet-to-Neumann map) is an FIO, and the scattering relation of the sound speed as well as certain X-ray transforms of the coefficients can be computed from the canonical relation and the symbol of this FIO.
\end{enumerate}
%Practical details:

%\vspace{-10pt}

These notes are organized as follows. In Section \ref{sec_psdo}, we will motivate the theory of $\Psi$DOs and discuss some of its properties without giving proofs. Section \ref{sec_fourier_integral_operators} will continue with a brief introduction to wave front sets and FIOs (again with no proofs). The rest of the notes is concerned with applications to inverse problems. Section \ref{section_ray_transform} considers the Radon transform in $\mR^2$ and its normal operator, and describes what kind of information about the singularities of $f$ can be stably recovered from the Radon transform. Sections \ref{sec_gelfand} and \ref{sec_calderon_boundary} discuss the Gel'fand and Calder\'on problems, and prove results related to recovering X-ray transforms or boundary determination. The treatment is motivated by $\Psi$DO and FIO theory, but we give direct and (in principle) elementary proofs based on a quasimode constructions. The results discussed in these notes are classical. For more recent results, we refer to the surveys \cite{IlmavirtaMonard, KrishnanQuinto, Lassas, Uhlmann2014}.

\begin{comment}
\subsection*{References}

References for basic microlocal analysis (roughly in ascending order of difficulty) include the online lecture notes 

\begin{itemize}
\item 
J. Wunsch: Microlocal analysis and evolution equations
\item 
N. Lerner: A first course on pseudo-differential operators
\end{itemize}
and the textbooks 
\begin{itemize}
\item 
M. Wong: An introduction to pseudo-differential operators
\item 
X. Saint-Raymond: Elementary introduction to the theory of pseudodifferential operators
\item 
G. Eskin: Lectures on linear partial differential equations
\item 
A. Grigis, J. Sj\"ostrand: Microlocal analysis for differential operators - an introduction
\item 
M. Shubin: Pseudodifferential operators and spectral theory
\item 
M. Taylor: Partial differential equations, vol. II
\item 
L. H\"ormander: The analysis of linear partial differential operators, vols. I-IV
\end{itemize}
\end{comment}

\subsection*{Notation}
%\addcontentsline{toc}{chapter}{Notation}

%We will write $\mR$, $\mC$, and $\mZ$ for the real numbers, complex numbers, and the integers, respectively. $\mR_+$ will be the set of positive real numbers and $\mZ_+$ the set of positive integers, with $\mN = \mZ_+ \cup \{0\}$ the set of natural numbers. For vectors $x$ in $\mR^n$ the expression $\abs{x}$ denotes the Euclidean length. %, while for vectors $k$ in $\mZ^n$ we write $\abs{k} = \sum_{i=1}^n \abs{k_i}$. We will also use the notation $\br{x} = (1+\abs{x}^2)^{1/2}$.

%If $\Omega \subset \mR^n$ is an open set then $C(\Omega)$ is the space of bounded continuous complex functions on $\Omega$, normed with the infinity norm. Support conditions are indicated by a subscript. Particularly, $C_c(\Omega)$ is the space of compactly supported continuous functions, and $C_0(\Omega)$ the space of continuous functions vanishing at infinity. Both these spaces are also given the topology of uniform convergence.

We will use multi-index notation. Let $\mN_0 = \{ 0, 1, 2, \ldots \}$ be the set natural numbers. Then $\mN_0^n$ consists of all $n$-tuples $\alpha = (\alpha_1,\ldots,\alpha_n)$ where the $\alpha_j$ are nonnegative integers. Such an $n$-tuple $\alpha$ is called a \emph{multi-index}. We write $|\alpha| = \alpha_1 + \ldots + \alpha_n$ and 
$\xi^{\alpha} = \xi_1^{\alpha_1} \cdots \xi_n^{\alpha_n}$ for $\xi \in \mR^n$. For partial derivatives, we will write 
\[
\p_j = \frac{\p}{\p x_j}, \qquad D_j = \frac{1}{i} \p_j, \qquad D = \frac{1}{i} \nabla, \qquad D^{\alpha} = D_1^{\alpha_1} \cdots D_n^{\alpha_n}.
\end{equation*}

If $\Omega \subset \mR^n$ is a bounded domain with $C^{\infty}$ boundary, we denote by $C^{\infty}(\ol{\Omega})$ the set of infinitely differentiable functions in $\Omega$ whose all derivatives extend continuously to $\ol{\Omega}$. The space $C^{\infty}_c(\Omega)$ consist of $C^{\infty}$ functions having compact support in $\Omega$. The standard $L^2$ based Sobolev spaces are denoted by $H^s(\mR^n)$ with norm $\norm{f}_{H^s(\mR^n)} = \norm{(1+\abs{\xi}^2)^{s/2} \hat{f}}_{L^2(\mR^n)}$, with $\hat{f}$ denoting the Fourier transform. We also write $\norm{f}_{W^{k,\infty}} = \sum_{\abs{\alpha} \leq k} \norm{D^{\alpha} f}_{L^{\infty}}$. The notation $A \lesssim B$ means that $A \leq CB$ for some uniform (with respect to the relevant parameters) constant $C$. In general, all coefficients, boundaries etc are assumed to be $C^{\infty}$ for ease of presentation.

%We also write $\alpha! = \alpha_1! \cdots \alpha_n!$.

%Throughout these notes we will apply the \emph{Einstein summation convention}: repeated indices in lower and upper position are summed. For instance, the expression 
%\begin{equation*}
%a_{jkl} b^j c^k
%\end{equation*}
%is shorthand for 
%\begin{equation*}
%\sum_{j, k} a_{jkl} b^j c^k.
%\end{equation*}
%The summation indices run typically from $1$ to $n$, where $n$ is the dimension of the manifold in question.

\section{Pseudodifferential operators} \label{sec_psdo}

In this minicourse we will try to give a very brief idea of the different points of view to microlocal analysis mentioned in the introduction (and repeated below), as 
\begin{enumerate}
\item[(1)] 
a kind of "variable coefficient Fourier analysis" for solving variable coefficient PDEs; or 
\item[(2)] 
a theory of $\Psi$DOs and FIOs; or 
\item[(3)]  
a phase space (or time-frequency) approach to studying functions, operators and their singularities (wave front sets).
\end{enumerate}

In this section we will discuss (1) and (2) in the context of $\Psi$DOs (we will continue with (2) and (3) in the context of FIOs in Section \ref{sec_fourier_integral_operators}). The treatment is mostly formal and we will give no proofs whatsoever. A complete reference for the results in this section is \cite[Section 18.1]{Hormander}.

\subsection{Constant coefficient PDEs} \label{subsec_constant_coefficient_pdes}

We recall the following facts about the Fourier transform (valid for sufficiently nice functions): 
\begin{enumerate}
\item[1.]
If $u$ is a function in $\mR^n$, its \emph{Fourier transform} $\hat{u} = \mF u$ is the function 
\[
\hat{u}(\xi) := \int_{\mR^n} e^{-ix \cdot \xi} u(x) \,dx, \qquad \xi \in \mR^n.
\]
\item[2.]
The Fourier transform converts derivatives to polynomials (this is why it is useful for solving PDEs):
\[
(D_j u)\ehat(\xi) = \xi_j \hat{u}(\xi).
\]
\item[3.]
A function $u$ can be recovered from $\hat{u}$ by the Fourier inversion formula $u = \mF^{-1}\{Ê\hat{u} \}$, where 
\[
\mF^{-1} v(x) := (2\pi)^{-n} \int_{\mR^n} e^{ix \cdot \xi} v(\xi) \,d\xi
\]
is the \emph{inverse Fourier transform}.
\end{enumerate}

As a motivating example, let us solve formally (i.e.\ without worrying about how to precisely justify each step) the equation 
\[
-\Delta u = f \text{ in $\mR^n$}.
\]
This is a constant coefficient PDE, and such equations can be studied with the help of the Fourier transform. We formally compute 
\begin{align}
-\Delta u = f &\ \Longleftrightarrow \ \abs{\xi}^2 \hat{u}(\xi) = \hat{f}(\xi) \notag \\
 &\ \Longleftrightarrow \ \hat{u}(\xi) = \frac{1}{\abs{\xi}^2} \hat{f}(\xi) \notag \\
 &\ \Longleftrightarrow \ u(x) = \mF^{-1} \left\{Ê\frac{1}{\abs{\xi}^2} \hat{f}(\xi) \right\} = (2\pi)^{-n} \int_{\mR^n} e^{ix \cdot \xi} \frac{1}{\abs{\xi}^2} \hat{f}(\xi) \,d\xi. \label{u_laplace_formal}
\end{align}
The same formal argument applies to a general constant coefficient PDE 
\[
a(D)u = f \text{ in $\mR^n$}, \qquad a(D) = \sum_{\abs{\alpha} \leq m} a_{\alpha} D^{\alpha},
\]
where $a_{\alpha} \in \mC$. Then $(a(D)u) \ehat(\xi) = a(\xi) \hat{u}(\xi)$ where $a(\xi) = \sum_{\abs{\alpha} \leq m} a_{\alpha} \xi^{\alpha}$ is the \emph{symbol} of $a(D)$. Moreover, one has 
\begin{equation} \label{adu_formula}
a(D) u(x) = \mF^{-1} \left\{Êa(\xi) \hat{u}(\xi) \right\} = (2\pi)^{-n} \int_{\mR^n} e^{ix \cdot \xi} a(\xi) \hat{f}(\xi) \,d\xi.
\end{equation}
The argument leading to \eqref{u_laplace_formal} gives a formal solution of $a(D)u = f$:
\begin{equation} \label{u_ad_solution_formula}
u(x) = \mF^{-1} \left\{Ê\frac{1}{a(\xi)} \hat{u}(\xi) \right\} = (2\pi)^{-n} \int_{\mR^n} e^{ix \cdot \xi} \frac{1}{a(\xi)} \hat{f}(\xi) \,d\xi.
\end{equation}
Thus formally $a(D)u = f$ can be solved by dividing by the symbol $a(\xi)$ on the Fourier side. Of course, to make this precise one would need to show that the division by $a(\xi)$ (which may have zeros) is somehow justified.

\subsection{Variable coefficient PDEs} \label{subsec_variable_coefficient_pde}

We now try to use a similar idea to solve the variable coefficient PDE 
\[
Au = f \text{ in $\mR^n$}, \qquad A = a(x,D) = \sum_{\abs{\alpha} \leq m} a_{\alpha}(x) D^{\alpha},
\]
where $a_{\alpha}(x) \in C^{\infty}(\mR^n)$ and $D^{\beta} a_{\alpha} \in L^{\infty}(\mR^n)$ for all multi-indices $\alpha, \beta$. Since the coefficients $a_{\alpha}$ depend on $x$, Fourier transforming the equation $Au = f$ is not immediately helpful. However, we can compute an analogue of \eqref{adu_formula}:
\begin{align}
Au(x) &= A \left[ \mF^{-1} \{Ê\hat{u}(\xi) \} \right] \notag \\
 &= \sum_{\abs{\alpha} \leq m} a_{\alpha}(x) D^{\alpha} \left[ (2\pi)^{-n} \int_{\mR^n} e^{ix \cdot \xi} \hat{u}(\xi) \,d\xi \right] \notag \\
 &= (2\pi)^{-n} \int_{\mR^n} e^{ix \cdot \xi} \left[ \sum_{\abs{\alpha} \leq m} a_{\alpha}(x) \xi^{\alpha} \right] \hat{u}(\xi) \,d\xi \notag \\
 &= (2\pi)^{-n} \int_{\mR^n} e^{ix \cdot \xi} a(x,\xi) \hat{u}(\xi) \,d\xi \label{aux_differential_operator_formula}
\end{align}
where 
\begin{equation} \label{a_differential_symbol}
a(x,\xi) := \sum_{\abs{\alpha} \leq m} a_{\alpha}(x) \xi^{\alpha}
\end{equation}
is the (full) \emph{symbol} of $A = a(x,D)$.

Now, we could try to obtain a solution to $a(x,D) u = f$ in $\mR^n$ by dividing by the symbol $a(x,\xi)$ as in \eqref{u_ad_solution_formula}:
\[
u(x) = (2\pi)^{-n} \int_{\mR^n} e^{ix \cdot \xi} \frac{1}{a(x,\xi)} \hat{f}(\xi) \,d\xi.
\]
Again, this is only formal since the division by $a(x,\xi)$ needs to be justified. However, this can be done in a certain sense if $A$ is \emph{elliptic}:

\begin{definition}
The \emph{principal symbol} (i.e.\ the part containing the highest order derivatives) of the differential operator $A = a(x,D)$ is 
\[
\sigma_{\mathrm{pr}}(A) := \sum_{\abs{\alpha} = m} a_{\alpha}(x) \xi^{\alpha}.
\]
We say that $A$ is \emph{elliptic} if its principal symbol is nonvanishing for $\xi \neq 0$.
\end{definition}

A basic result of microlocal analysis states that the function 
\[
u_1(x) := (2\pi)^{-n} \int_{\mR^n} e^{ix \cdot \xi} b(x,\xi) \hat{f}(\xi) \,d\xi
\]
with 
\begin{equation} \label{b_division_approximate}
b(x,\xi) := \frac{1-\psi(\xi)}{a(x,\xi)},
\end{equation}
where $\psi \in C^{\infty}_c(\mR^n)$ is a cutoff with $\psi(\xi) = 1$ in a sufficiently large neighborhood of $\xi=0$ (so that $a(x,\xi)$ does not vanish outside this neighborhood), is an \emph{approximate solution} of $Au = f$ in the sense that 
\[
Au_1 = f + f_1
\]
where $f_1$ is one derivative smoother than $f$. Moreover, it is possible to construct an approximate solution $u_{\mathrm{app}}$ so that 
\[
Au_{\mathrm{app}} = f + r, \qquad r \in C^{\infty}(\mR^n).
\]

\subsection{Pseudodifferential operators} \label{subseq_pseudodifferential}

In analogy with the formula \eqref{aux_differential_operator_formula}, a \emph{pseudodifferential operator} ($\Psi$DO) is an operator $A$ of the form 
\begin{equation} \label{a_pseudodifferential_definition_one}
Au(x) = (2\pi)^{-n} \int_{\mR^n} e^{ix \cdot \xi} a(x,\xi) \hat{u}(\xi) \,d\xi
\end{equation}
where $a(x,\xi)$ is a \emph{symbol} with certain properties. The most standard symbol class $S^m = S^m_{1,0}(\mR^n)$ is defined as follows:

\begin{definition}
The symbol class $S^m$ consists of functions $a \in C^{\infty}(\mR^n \times \mR^n)$ such that for any $\alpha, \beta \in \mN_0^n$ there is $C_{\alpha,\beta} > 0$ with 
\[
\abs{\p_x^{\alpha} \p_{\xi}^{\beta} a(x,\xi)} \leq C_{\alpha,\beta} (1+\abs{\xi})^{m-\abs{\beta}}, \qquad \xi \in \mR^n.
\]
If $a \in S^m$, the corresponding $\Psi$DO $A = \mathrm{Op}(a)$ is defined by \eqref{a_pseudodifferential_definition_one}. We denote by $\Psi^m$ the set of $\Psi$DOs corresponding to $S^m$.
\end{definition}

Note that symbols in $S^m$ behave roughly like polynomials of order $m$ in the $\xi$-variable. In particular, the symbols $a(x,\xi)$ in \eqref{a_differential_symbol} belong to $S^m$ and the corresponding differential operators $a(x,D)$ belong to $\Psi^m$. Moreover, if $a(x,D)$ is elliptic, then the symbol $b(x,\xi) = \frac{1-\psi(\xi)}{a(x,\xi)}$ as in \eqref{b_division_approximate} belongs to $S^{-m}$.  Thus the class of $\Psi$DOs is large enough to include  differential operators as well as approximate inverses of elliptic operators. Also normal operators of the X-ray transform or Radon transform in $\mR^n$ are $\Psi$DOs.

\begin{remark}[Homogeneous symbols]
We saw in Section \ref{subsec_constant_coefficient_pdes} that the elliptic operator $-\Delta$ has the inverse 
\[
G: f \mapsto \mF^{-1} \left\{ \frac{1}{\abs{\xi}^2} \hat{f}(\xi) \right\}.
\]
The symbol $\frac{1}{\abs{\xi}^2}$ is not in $S^{-2}$, since it is not smooth near $0$. However, one often thinks of $G$ as a $\Psi$DO by writing 
\[
G = G_1 + G_2, \quad G_1 := \mF^{-1} \left\{ \frac{1-\psi(\xi)}{\abs{\xi}^2} \hat{f}(\xi) \right\}, \quad G_2 := \mF^{-1} \left\{ \frac{\psi(\xi)}{\abs{\xi}^2} \hat{f}(\xi) \right\},
\]
where $\psi \in C^{\infty}_c(\mR^n)$ satisfies $\psi = 1$ near $0$. Now $G_1$ is a $\Psi$DO in  $\Psi^{-2}$, since $\frac{1-\psi(\xi)}{\abs{\xi}^2} \in S^{-2}$, and $G_2$ is smoothing in the sense that it maps any $L^1$ function into a $C^{\infty}$ function (at least if $n \geq 3$).

In general, in $\Psi$DO theory smoothing operators are considered to be negligible (since at least they do not introduce new singularities), and many computations in $\Psi$DO calculus are made only modulo smoothing error terms. In this sense one often views $G$ as a $\Psi$DO by identifying it with $G_1$. The same kind of identification is done for operators whose symbol $a(x,\xi)$ is homogenous of some order $m$ in $\xi$. More generally one can consider \emph{polyhomogeneous} symbols $b \in S^m$ having the form 
\[
b(x,\xi) \sim \sum_{j=0}^{\infty} b_{m-j}(x,\xi)
\]
where each $b_{m-j}$ is homogeneous of order $m-j$ in $\xi$, and $\sim$ is a certain asympotic summation. Corresponding $\Psi$DOs are called \emph{classical $\Psi$DOs}.
\end{remark}

It is very important that one can compute with $\Psi$DOs in much the same way as with differential operators. One often says that $\Psi$DOs have a \emph{calculus}. The following theorem lists typical rules of computation (it is instructive to think first why such rules are valid for differential operators):

\begin{theorem}[$\Psi$DO calculus] \label{thm_psdo_calculus_basic} $\phantom{a}$
\begin{enumerate}
\item[(a)] 
(Principal symbol) There is a one-to-one correspondence between operators in $\Psi^m$ and (full) symbols in $S^m$, and each operator $A \in \Psi^m$ has a well defined \emph{principal symbol} $\sigma_{\mathrm{pr}}(A)$. The principal symbol may be computed by testing $A$ against highly oscillatory functions\footnote{This is valid if $A$ is a classical $\Psi$DO.}:
\begin{equation} \label{principal_symbol_formula}
\sigma_{\mathrm{pr}}(A)(x,\xi) = \lim_{\lambda \to \infty} \lambda^{-m} e^{-i\lambda x \cdot \xi} A(e^{i\lambda x \cdot \xi});
\end{equation}
\item[(b)] 
(Composition) If $A \in \Psi^m$ and $B \in \Psi^{m'}$, then $AB \in \Psi^{m+m'}$ and $\sigma_{\mathrm{pr}}(AB) = \sigma_{\mathrm{pr}}(A) \sigma_{\mathrm{pr}}(B)$;
\item[(c)] 
(Sobolev mapping properties) Each $A \in \Psi^m$ is a bounded operator $H^s(\mR^n) \to H^{s-m}(\mR^n)$ for any $s \in \mR$;
\item[(d)] 
(Elliptic operators have approximate inverses) If $A \in \Psi^m$ is elliptic, there is $B \in \Psi^{-m}$ so that $AB = \mathrm{Id} + K$ and $BA = \mathrm{Id} + L$ where $K, L \in \Psi^{-\infty}$, i.e.\ $K, L$ are \emph{smoothing} (they map any $H^{-s}$ function to $H^t$ for any $t$, hence also to $C^{\infty}$ by Sobolev embedding).
\end{enumerate}
\end{theorem}

The above properties are valid in the standard $\Psi$DO calculus in $\mR^n$. However, motivated by different applications, $\Psi$DOs have been considered in various other settings. Each of these settings comes with an associated calculus whose rules of computation are similar but adapted to the situation at hand. Examples of different settings for $\Psi$DO calculus include 
\begin{enumerate}
\item 
open sets in $\mR^n$ (local setting);
\item 
compact manifolds without boundary, possibly acting on sections of vector bundles; 
\item 
compact manifolds with boundary (transmission condition / Boutet de Monvel calculus);
\item 
non-compact manifolds (e.g.\ Melrose scattering calculus); and 
\item 
operators with a small or large parameter (semiclassical calculus).
\end{enumerate}

\section{Wave front sets and Fourier integral operators} \label{sec_fourier_integral_operators}

For a reference to wave front sets, see \cite[Chapter 8]{Hormander}. Sobolev wave front sets are considered in \cite[Section 18.1]{Hormander}. FIOs are discussed in \cite[Chapter 25]{Hormander}.

\subsection{The role of singularities}

We first discuss the singular support of $u$, which consists of those points $x_0$ such that $u$ is not a smooth function in any neighborhood of $x_0$. We also consider the Sobolev singular support, which also measures the "strength" of the singularity (in the $L^2$ Sobolev scale).

\begin{definition}[Singular support]
We say that a function or distribution $u$ is \emph{$C^{\infty}$ (resp.\ $H^{\alpha}$) near $x_0$} if there is $\varphi \in C^{\infty}_c(\mR^n)$ with $\varphi = 1$ near $x_0$ such that $\varphi u$ is in $C^{\infty}(\mR^n)$ (resp.\ in $H^{\alpha}(\mR^n)$). We define 
\begin{align*}
\mathrm{sing\,supp}(u) &= \mR^n \setminus \{ x_0 \in \mR^n \,;\, \text{$u$ is $C^{\infty}$ near $x_0$} \}, \\
\mathrm{sing\,supp}^{\alpha}(u) &= \mR^n \setminus \{ x_0 \in \mR^n \,;\, \text{$u$ is $H^{\alpha}$ near $x_0$} \}.
\end{align*}
\end{definition}

\begin{example} \label{example_piecewise_constant}
Let $D_1, \ldots, D_N$ be bounded domains with $C^{\infty}$ boundary in $\mR^n$ so that $\ol{D}_j \cap \ol{D}_k = \emptyset$ for $j \neq k$, and define 
\[
u = \sum_{j=1}^N c_j \chi_{D_j}
\]
where $c_j \neq 0$ are constants, and $\chi_{D_j}$ is the characteristic function of $D_j$. Then 
\[
\mathrm{sing\,supp}^{\alpha}(u) = \emptyset \text{ for $\alpha < 1/2$}
\]
since $u \in H^{\alpha}$ for $\alpha < 1/2$, but 
\[
\mathrm{sing\,supp}^{\alpha}(u) = \bigcup_{j=1}^N \p D_j \text{ for $\alpha \geq 1/2$}
\]
since $u$ is not $H^{1/2}$ near any boundary point. Thus in this case the singularities of $u$ are exactly at the points where $u$ has a jump discontinuity, and their strength is precisely $H^{1/2}$. Knowing the singularities of $u$ can already be useful in applications. For instance, if $u$ represents some internal medium properties in medical imaging, the singularities of $u$ could determine the location of interfaces between different tissues. On the other hand, if $u$ represents an image, then the singularities in some sense determine the "sharp features" of the image.
\end{example}

Next we discuss the \emph{wave front set} which is a more refined notion of a singularity. For example, if $f = \chi_D$ is the characteristic function of a bounded strictly convex $C^{\infty}$ domain $D$ and if $x_0 \in \p D$, one could think that $f$ is in some sense smooth in tangential directions at $x_0$ (since $f$ restricted to a tangent hyperplane is identically zero, except possibly at $x_0$), but that $f$ is not smooth in normal directions at $x_0$ since in these directions there is a jump. The wave front set is a subset of $T^* \mR^n \setminus 0$, the cotangent space with the zero section removed:
\[
T^* \mR^n \setminus 0 := \{ (x,\xi) \,;\, x, \xi \in \mR^n, \xi \neq 0 \}.
\]

\begin{definition}[Wave front set]
Let $u$ be a distribution in $\mR^n$. We say that $u$ is (microlocally) \emph{$C^{\infty}$ (resp.\ $H^{\alpha}$) near $(x_0,\xi_0)$} if there exist $\varphi \in C^{\infty}_c(\mR^n)$ with $\varphi = 1$ near $x_0$ and $\psi \in C^{\infty}(\mR^n \setminus \{0\})$ so that $\psi = 1$ near $\xi_0$ and $\psi$ is homogeneous of degree $0$, such that 
\[
\text{for any $N$ there is $C_N > 0$ so that $\psi(\xi) ( \varphi u ) \ehat(\xi) \leq C_N(1+\abs{\xi})^{-N}$}
\]
(resp.\ $\mF^{-1} \{ \psi(\xi) ( \varphi u ) \ehat(\xi) \} \in H^{\alpha}(\mR^n)$). The \emph{wave front set} $WF(u)$ (resp.\ \emph{$H^{\alpha}$ wave front set} $WF^{\alpha}(u)$) consists of those points $(x_0,\xi_0)$ where $u$ is not microlocally $C^{\infty}$ (resp.\ $H^{\alpha}$).
\end{definition}

\begin{example}
The wave front set of the function $u$ in Example \ref{example_piecewise_constant} is 
\[
WF(u) = \bigcup_{j=1}^N N^*(D_j)
\]
where $N^*(D_j)$ is the conormal bundle of $D_j$,
\[
N^*(D_j) := \{ (x,\xi) \,;\, \text{$x \in \p D_j$ and $\xi$ is normal to $\p D_j$ at $x$} \}.
\]
\end{example}

The wave front set describes singularities more precisely than the singular support, since one always has 
\begin{equation} \label{wave_front_set_projection}
\pi(WF(u)) = \mathrm{sing\,supp}(u)
\end{equation}
where $\pi: (x,\xi) \mapsto x$ is the projection to $x$-space.

It is an important fact that applying a $\Psi$DO to a function or distribution never creates new singularities:

\begin{theorem}[Pseudolocal/microlocal property of $\Psi$DOs] \label{thm_psdo_microlocal_property}
Any $A \in \Psi^m$ has the pseudolocal property 
\begin{align*}
\mathrm{sing\,supp}(Au) &\subset \mathrm{sing\,supp}(u), \\
\mathrm{sing\,supp}^{\alpha-m}(Au) &\subset \mathrm{sing\,supp}^{\alpha}(u)
\end{align*}
and the microlocal property 
\begin{align*}
WF(Au) &\subset WF(u), \\
WF^{\alpha-m}(Au) &\subset WF^{\alpha}(u).
\end{align*}
\end{theorem}

Elliptic operators are those that completely preserve singularities:

\begin{theorem} \label{thm_elliptic_regularity_singsupp}
(Elliptic regularity) Let $A \in \Psi^m$ be elliptic. Then, for any $u$, 
\begin{align*}
\mathrm{sing\,supp}(Au) &= \mathrm{sing\,supp}(u), \\
WF(Au) &= WF(u).
\end{align*}
Thus any solution $u$ of $Au = f$ is singular precisely at those points where $f$ is singular. There are corresponding statements for Sobolev singularities.
\end{theorem}
\begin{proof}
First note that by Theorem \ref{thm_psdo_microlocal_property}, 
\[
WF(Au) \subset WF(u).
\]
Conversely, since $A \in \Psi^m$ is elliptic, by Theorem \ref{thm_psdo_calculus_basic}(d) there is $B \in \Psi^{-m}$ so that 
\[
BA = \mathrm{Id} + L, \qquad L \in \Psi^{-\infty}.
\]
Thus for any $u$ one has 
\[
u + Lu = BAu.
\]
Since $L$ is smoothing, $Lu \in C^{\infty}$, which implies that $u = BAu$ modulo $C^{\infty}$. Thus it follows that 
\[
WF(u) = WF(BAu) \subset WF(Au).
\]
Thus $WF(Au) = WF(u)$. The claim for singular supports follows by \eqref{wave_front_set_projection}.
\end{proof}

\subsection{Fourier integral operators}

We have seen in Section \ref{subseq_pseudodifferential} that the class of pseudodifferential operators includes approximate inverses of elliptic operators. In order to handle approximate inverses for hyperbolic and transport equations, it is required to work with a larger class of operators.

\begin{motivation}
Consider the initial value problem for the wave equation,
\begin{gather*}
(\p_t^2 - \Delta) u(x,t) = 0 \text{ in $\mR^n \times (0,\infty)$}, \\
u(x,0) = f(x), \qquad \p_t u(x,0) = 0.
\end{gather*}
This is again a constant coefficient PDE, and we will solve this formally by taking the Fourier transform in space, 
\[
\tilde{u}(\xi,t) := \int_{\mR^n} e^{-ix \cdot \xi} u(x,t) \,dx, \qquad \xi \in \mR^n.
\]
After taking Fourier transforms in space, the above equation becomes 
\begin{gather*}
(\p_t^2 + \abs{\xi}^2) \tilde{u}(\xi,t) = 0 \text{ in $\mR^n \times (0,\infty)$}, \\
\tilde{u}(\xi,0) = \hat{f}(\xi), \qquad \p_t \tilde{u}(\xi,0) = 0.
\end{gather*}
For each fixed $\xi$ this is an ODE in $t$, and the solution is 
\[
\tilde{u}(\xi,t) = \cos(t\abs{\xi}) \hat{f}(\xi) = \frac{1}{2} (e^{i t \abs{\xi}} + e^{-i t \abs{\xi}}) \hat{f}(\xi).
\]
Taking inverse Fourier transforms in space, we obtain 
\begin{equation} \label{uxt_classical_representation}
u(x,t) = \frac{1}{2} \sum_{\pm} (2\pi)^{-n} \int_{\mR^n} e^{i(x \cdot \xi \pm t\abs{\xi})} \hat{f}(\xi) \,d\xi.
\end{equation}
\end{motivation}

Generalizing \eqref{uxt_classical_representation}, we can consider operators of the form 
\begin{equation} \label{a_fio_definition}
Au(x) = (2\pi)^{-n} \int_{\mR^n} e^{i \varphi(x,\xi)} a(x,\xi) \hat{u}(\xi) \,d\xi
\end{equation}
where $a(x,\xi)$ is a symbol (for instance in $S^m$), and $\varphi(x,\xi)$ is a real valued phase function. Such operators are examples of \emph{Fourier integral operators} (more precisely, FIOs whose canonical relation is locally the graph of a canonical transformation, see \cite[Section 25.3]{Hormander}). For $\Psi$DOs the phase function is always $\varphi(x,\xi) = x \cdot \xi$, but for FIOs the phase function can be quite general, though it is usually required to be homogeneous of degree $1$ in $\xi$, and to satisfy the non-degeneracy condition $\det(\p_{x_j \xi_k} \varphi) \neq 0$.

We will not go into precise definitions, but only remark that the class of FIOs includes pseudodifferential operators as well as approximate inverses of hyperbolic and transport operators (or more generally real principal type operators). There is a calculus for FIOs, analogous to the pseudodifferential calculus, under certain conditions in various settings. An important property of FIOs is that they, unlike pseudodifferential operators, can move singularities. This aspect will be discussed next.

\subsection{Propagation of singularities}

\begin{example} \label{example_wave_equation_propagation}
Let $t > 0$ be fixed, and consider the operators from \eqref{uxt_classical_representation}, 
\[
A_{\pm t} f(x) = (2\pi)^{-n} \int_{\mR^n} e^{i(x \cdot \xi \mp t\abs{\xi})} \hat{f}(\xi) \,d\xi.
\]
Then 
\[
u(x,t) = \frac{1}{2}(A_{+t}f(x) + A_{-t} f(x)).
\]
Using FIO theory, since the phase functions are $\varphi(x,\xi) = x \cdot \xi \mp t\abs{\xi}$, it follows that 
\[
WF(A_{\pm t} f) \subset \chi_{\pm t}(WF(f))
\]
where $\chi_{\pm t}$ is the \emph{canonical transformation} (i.e.\ diffeomorphism of $T^* \mR^n \setminus 0$ that preserves the symplectic structure) given by 
\[
\chi_{\pm t}(x,\xi) = (x \pm t\xi/\abs{\xi}, \xi).
\]
This means that the FIO $A_{\pm}$ takes a singularity $(x,\xi)$ of the initial data $f$ and moves it along the line through $x$ in direction $\pm \xi/\abs{\xi}$ to $(x \pm t\xi/\abs{\xi}, \xi)$. Thus singularities of solutions of the wave equation $(\p_t^2 - \Delta) u = 0$ propagate along straight lines with constant speed one.
\end{example}

\begin{remark}
In general, any FIO has an associated \emph{canonical relation} that describes what the FIO does to singularities. The canonical relation of the FIO $A$ defined in \eqref{a_fio_definition} is (see \cite[Section 25.3]{Hormander}) 
\[
C = \{ (x, \nabla_x \varphi(x,\xi), \nabla_{\xi} \varphi(x,\xi), \xi) \,;\, (x,\xi) \in T^* \mR^n \setminus 0 \},
\]
and $A$ moves singularities according to the rule 
\[
WF(Au) \subset C(WF(u))
\]
where 
\[
C(WF(u)) := \{ (x,\xi) \;;\, (x,\xi,y,\eta) \in C \text{ for some } (y,\eta) \in WF(u) \}.
\]
Using these formulas, it is easy to check that the canonical relation $C_{\pm}$ of $A_{\pm t}$ in Example \ref{example_wave_equation_propagation} is the graph of $\chi_{\pm t}$ in the sense that 
\[
C_{\pm} = \{ (\chi_{\pm t}(y,\eta), y,\eta ) \,;\, (y,\eta) \in T^* \mR^n \setminus 0 \}
\]
and one indeed has $WF(A_{\pm t} u) \subset C_{\pm}(WF(u)) = \chi_{\pm t}(WF(u))$.
\end{remark}

There is a far reaching extension of Example \ref{example_wave_equation_propagation}, which shows that the singularities of a solution of $Pu = 0$ propagate along certain curves in phase space (so called \emph{null bicharacteristic curves}) as long as $P$ has real valued principal symbol.

\begin{theorem}[Propagation of singularities] \label{thm_propagation_of_singularities}
Let $P \in \Psi^m$ have real principal symbol $p_m$ that is homogeneous of degree $m$ in $\xi$. If 
\[
Pu = f,
\]
then $WF(u) \setminus WF(f)$ is contained in the characteristic set $p_m^{-1}(0)$. Moreover, if $(x_0,\xi_0) \in WF(u) \setminus WF(f)$, then the whole null bicharacteristic curve $(x(t),\xi(t))$ through $(x_0,\xi_0)$ is in $WF(u) \setminus WF(f)$, where 
\begin{align*}
\dot{x}(t) &= \nabla_{\xi} p_m(x(t), \xi(t)), \\
\dot{\xi}(t) &= -\nabla_x p_m(x(t), \xi(t)).
\end{align*}
\end{theorem}

\begin{example} \label{example_bicharacteristic_wave}
We compute the null bicharacteristic curves for the wave operator $P = \frac{1}{2}(\Delta-\p_t^2)$. The principal symbol of $P$ is 
\[
p_2(x,t,\xi,\tau) = \frac{1}{2}(\tau^2 - \abs{\xi}^2)
\]
The characteristic set is 
\[
p_2^{-1}(0) = \{ (x,t,\xi,\tau) \,;\, \tau = \pm \abs{\xi} \}
\]
which consists of \emph{light-like} cotangent vectors on $\mR^{n+1}_{x,t}$. The equations for the null bicharacteristic curves are 
\begin{align*}
\dot{x}(s) &= -\xi(s), \\
\dot{t}(s) &= \tau(s), \\
\dot{\xi}(s) &= 0, \\
\dot{\tau}(s) &= 0.
\end{align*}
Thus, if $\abs{\xi_0} = 1$, then the null bicharacteristic curve through $(x_0,t_0,\xi_0, \pm 1)$ is 
\[
s \mapsto (x_0 - s \xi_0, t_0 \pm s , \xi_0, \pm 1)
\]
The result of Example \ref{example_wave_equation_propagation} may thus be interpreted so that singularities of solutions of the wave equation propagate along null bicharacteristic curves for the wave operator.
\end{example}

\section{The Radon transform in the plane} \label{section_ray_transform}

In this section we outline some applications of microlocal analysis to the study of the Radon transform in the plane. Similar ideas apply to X-ray and Radon transforms in higher dimensions and Riemannian manifolds as well. The microlocal approach to Radon transforms was introduced by Guillemin \cite{Guillemin}. We refer to \cite{Quinto}, \cite{KrishnanQuinto} and references therein for a more detailed treatment of the material in this section.

\subsection{Basic properties of the Radon transform}

The \emph{X-ray transform} $If$ of a function $f$ in $\mR^n$ encodes the integrals of $f$ over all straight lines, whereas the \emph{Radon transform} $Rf$ encodes the integrals of $f$ over $(n-1)$-dimensional planes. We will focus on the case $n=2$, where the two transforms coincide. There are many ways to parametrize the set of lines in $\mR^2$. We will parametrize lines by their direction vector $\omega$ and distance $s$ from the origin.

\begin{definition}
If $f \in C^{\infty}_c(\mR^2)$, the \emph{Radon transform} of $f$ is the function 
\begin{equation*}
Rf(s,\omega) := \int_{-\infty}^{\infty} f(s\omega^{\perp} + t\omega) \,dt, \quad s \in \mR, \ \omega \in S^1.
\end{equation*}
Here $\omega^{\perp}$ is the vector in $S^1$ obtained by rotating $\omega$ counterclockwise by $90^{\circ}$.
\end{definition}

There is a well-known relation between $Rf$ and the Fourier transform $\hat{f}$. We denote by $(Rf)\etilde(\,\cdot\,,\omega)$ the Fourier transform of $Rf$ with respect to $s$.

\begin{theorem} \label{thm_fourier_slice}
(Fourier slice theorem)
\begin{equation*}
(Rf)\etilde(\sigma,\omega) = \hat{f}(\sigma \omega^{\perp}).
\end{equation*}
\end{theorem}
\begin{proof}
Parametrizing $\mR^2$ by $y = s\omega^{\perp} + t\omega$, we have 
\begin{align*}
(Rf)\etilde(\sigma,\omega) &= \int_{-\infty}^{\infty} e^{-i\sigma s} \left[ \int_{-\infty}^{\infty} f(s\omega^{\perp} + t\omega) \,dt \right] \,ds = \int_{\mR^2} e^{-i\sigma y \cdot \omega^{\perp}} f(y) \,dy \\
 &= \hat{f}(\sigma \omega^{\perp}). \qedhere
\end{align*}
\end{proof}

This result gives the first proof of injectivity of the Radon transform:

\begin{corollary} \label{corollary_radon_injectivity}
If $f \in C^{\infty}_c(\mR^2)$ is such that $Rf \equiv 0$, then $f \equiv 0$.
\end{corollary}
\begin{proof}
If $Rf \equiv 0$, then  $\hat{f} \equiv 0$ by Theorem \ref{thm_fourier_slice} and consequently $f \equiv 0$.
\end{proof}

To obtain a different inversion method, and for later purposes, we will consider the adjoint of $R$. The formal adjoint of $R$ is the \emph{backprojection operator}\footnote{The formula for $R^*$ is obtained as follows: if $f \in C^{\infty}_c(\mR^2)$, $h \in C^{\infty}(\mR \times S^1)$ one has  
\begin{align*}
(Rf, h)_{L^2(\mR \times S^1)} &= \int_{-\infty}^{\infty} \int_{S^1} Rf(s,\omega) \ol{h(s,\omega)} \,d\omega \,ds \\
 &= \int_{-\infty}^{\infty} \int_{S^1} \int_{-\infty}^{\infty} f(s\omega^{\perp} + t\omega) \ol{h(s,\omega)} \,dt \,d\omega \,ds \\
 &= \int_{\mR^2} f(y) \left( \int_{S^1} \ol{h(y \cdot \omega^{\perp}, \omega)} \,d\omega \right) \,dy.
\end{align*}}
 \begin{equation*}
R^*: C^{\infty}(\mR \times S^1) \to C^{\infty}(\mR^2), \ \ R^* h(y) = \int_{S^1} h(y \cdot \omega^{\perp}, \omega) \,d\omega.
\end{equation*}

\begin{comment}
\begin{proposition}
(Fourier transform of $R^*$) Letting $\hat{\xi} = \frac{\xi}{\abs{\xi}}$, 
\begin{equation*}
(R^* h)\ehat(\xi) = \frac{2\pi}{\abs{\xi}} \left( \hat{h}(\abs{\xi},-\hat{\xi}^{\perp}) + \hat{h}(-\abs{\xi},\hat{\xi}^{\perp}) \right).
\end{equation*}
\end{proposition}
\begin{proof}
We will make a formal computation (which is not difficult to justify). Using again the parametrization $y = s\omega^{\perp} + t\omega$, 
\begin{align*}
(R^* h)\ehat(\xi) &= \int_{\mR^2} \int_{S^1} e^{-iy \cdot \xi} h(y \cdot \omega^{\perp},\omega) \,d\omega \,dy \\
 &= \int_{-\infty}^{\infty} \int_{-\infty}^{\infty} \int_{S^1} e^{-is \omega^{\perp} \cdot \xi} e^{-it\omega \cdot \xi} h(s,\omega) \,d\omega \,ds \,dt \\
 &= \int_{S^1} \hat{h}(\omega^{\perp} \cdot \xi,\omega) \left( \int_{-\infty}^{\infty} e^{-it\omega \cdot \xi} \,dt \right) \,d\omega.
\end{align*}
The quantity in the parentheses is just $\frac{2\pi}{\abs{\xi}} \delta_0(\omega \cdot \hat{\xi})$ where $\delta_0$ is the Dirac delta function at the origin. Since $\omega \cdot \hat{\xi}$ is zero exactly when $\omega = \pm \hat{\xi}^{\perp}$, the result follows.
\end{proof}

The Radon transform in $\mR^2$ satisfies the symmetry $Rf(-s,-\omega) = Rf(s,\omega)$, and the Fourier slice theorem implies 
\begin{equation*}
(R^* R f)\ehat(\xi) = \frac{4\pi}{\abs{\xi}} \widehat{Rf}(\abs{\xi},-\hat{\xi}^{\perp}) = \frac{4\pi}{\abs{\xi}} \hat{f}(\xi).
\end{equation*}
\end{comment}

The following result shows that the normal operator $R^* R$ is a classical $\Psi$DO of order $-1$ in $\mR^2$, and also gives an inversion formula.

\begin{theorem} \label{theorem_normal_operator_radon}
(Normal operator) One has 
\begin{equation*}
R^* R = 4\pi \abs{D}^{-1} = \mF^{-1} \left\{Ê\frac{4\pi}{\abs{\xi}} \mF(\,\cdot\,) \right\},
\end{equation*}
and $f$ can be recovered from $Rf$ by the formula 
\begin{equation*}
f = \frac{1}{4\pi} \abs{D} R^* R f.
\end{equation*}
\end{theorem}
\begin{remark}
Above we have written, for $\alpha \in \mR$, 
\[
\abs{D}^{\alpha} f := \mF^{-1} \{ \abs{\xi}^{\alpha} \hat{f}(\xi) \}.
\]
The notation $(-\Delta)^{\alpha/2} = \abs{D}^{\alpha}$ is also used.
\end{remark}
\begin{proof}
The proof is based on computing $(Rf, Rg)_{L^2(\mR \times S^1)}$ using the Parseval identity, Fourier slice theorem, symmetry and polar coordinates:
\begin{align*}
(R^* R f, g)_{L^2(\mR^2)} &= (Rf, Rg)_{L^2(\mR \times S^1)} \\
 &= \int_{S^1} \left[ \int_{-\infty}^{\infty} (Rf)(s,\omega) \ol{(Rg)(s,\omega)} \,ds \right] \,d\omega \\
 &= \frac{1}{2\pi} \int_{S^1} \left[ \int_{-\infty}^{\infty} (Rf)\etilde(\sigma,\omega) \ol{(Rg)\etilde(\sigma,\omega)} \right] \,d\sigma \,d\omega \\
 &= \frac{1}{2\pi} \int_{S^1} \left[ \int_{-\infty}^{\infty}  \hat{f}(\sigma \omega^{\perp})  \ol{\hat{g}(\sigma \omega^{\perp})} \right] \,d\sigma \,d\omega \\
 &= \frac{2}{2\pi} \int_{S^1} \left[ \int_{0}^{\infty}  \hat{f}(\sigma \omega^{\perp}) \ol{\hat{g}(\sigma \omega^{\perp})} \right] \,d\sigma \,d\omega \\
 &= \frac{2}{2\pi} \int_{\mR^2} \frac{1}{\abs{\xi}} \hat{f}(\xi) \ol{\hat{g}(\xi)} \,d\xi \\
 &= ( 4\pi \mF^{-1} \left\{ \frac{1}{\abs{\xi}} \hat{f}(\xi) \right\}, g)_{L^2(\mR^2)}. \qedhere
\end{align*}
\end{proof}

The same argument, based on computing $(\abs{D_s}^{1/2} Rf, \abs{D_s}^{1/2} Rg)_{L^2(\mR \times S^1)}$ instead of $(Rf, Rg)_{L^2(\mR \times S^1)}$, leads to the famous \emph{filtered backprojection} (FBP) inversion formula:
\[
f = \frac{1}{4\pi} R^* \abs{D_s} R f
\]
where $\abs{D_s} Rf = \mF^{-1} \{Ê\abs{\sigma} (Rf) \etilde \}$. This formula is efficient to implement and gives good reconstructions when one has complete X-ray data and relatively small noise, and hence FBP (together with its variants) has been commonly used in X-ray CT scanners.

However, if one is mainly interested in the singularities (i.e.\ jumps or sharp features) of the image, it is possible to use the even simpler \emph{backprojection method}: just apply the backprojection operator $R^*$ to the data $Rf$. Since $R^* R$ is an elliptic $\Psi$DO, Theorem \ref{thm_elliptic_regularity_singsupp} guarantees that the singularities are recovered:
\[
\mathrm{sing\,supp}(R^* R f) = \mathrm{sing\,supp}(f).
\]
Moreover, since $R^* R$ is a $\Psi$DO of order $-1$, hence smoothing of order $1$, one expects that $R^* R f$ gives a slightly blurred version of $f$ where the main singularities should still be visible.

\subsection{Visible singularities}

There are various imaging situations where complete X-ray data (i.e.\ the function $Rf(s,\omega)$ for all $s$ and $\omega$) is not available. This is the case for limited angle tomography (e.g.\ in luggage scanners at airports, or dental applications), region of interest tomography, or exterior data tomography. In such cases explicit inversion formulas such as FBP are usually not available, but microlocal analysis (for related normal operators or FIOs) still provides a powerful paradigm for predicting which singularities can be recovered stably from the measurements. 

%The last result is an example of an explicit inversion method for the Radon transform in the Euclidean plane, based on the Fourier transform. Similar methods are available for the Radon transform on manifolds with many symmetries where variants of the Fourier transform exist (see \cite{helgason} and other books of Helgason for results of this type). However, for manifolds which do not have symmetries, such as small perturbations of the Euclidean metric, explicit transforms are usually not available and other inversion methods are required.

We will try to explain this paradigm a little bit more, starting with an example:

\begin{example}
Let $f$ be the characteristic function of the unit disc $\mD$, i.e.\ $f(x) = 1$ if $\abs{x} \leq 1$ and $f(x) = 0$ for $\abs{x} > 1$. Then $f$ is singular precisely on the unit circle (in normal directions). We have 
\[
Rf(s,\omega) = \left\{ \begin{array}{cl} 2\sqrt{1-s^2}, & s \leq 1, \\[5pt] 0, & s > 1. \end{array} \right.
\]
Thus $Rf$ is singular precisely at those points $(s,\omega)$ with $\abs{s}=1$, which correspond to those lines that are tangent to the unit circle.
\end{example}

There is a similar relation between the singularities of $f$ and $Rf$ in general, and this is explained by microlocal analysis:

\begin{theorem}
The operator $R$ is an elliptic FIO of order $-1/2$. There is a precise relationship between the singularities of $f$ and singularities of $Rf$.
\end{theorem}

We will not spell out the precise relationship here, but only give some consequences. It will be useful to think of the Radon transform as defined on the set of (non-oriented) lines in $\mR^2$. If $\mA$ is an open subset of lines in $\mR^2$, we consider the Radon transform $Rf|_{\mA}$ restricted to lines in $\mA$. Recovering $f$ (or some properties of $f$) from $Rf|_{\mA}$ is a \emph{limited data} tomography problem. Examples:

\begin{itemize}
\item 
If $\mA = \{ \text{lines not meeting $\ol{\mD}$} \}$, then $Rf|_{\mA}$ is called \emph{exterior data}.
\item 
If $0 < a < \pi/2$ and $\mA = \{ \text{lines whose angle with $x$-axis is $< a$} \}$, then $Rf|_{\mA}$ is called \emph{limited angle data}.
\end{itemize}

It is known that any $f \in C^{\infty}_c(\mR^2 \setminus \ol{D})$ is uniquely determined by exterior data (Helgason support theorem), and any $f \in C^{\infty}_c(\mR^2)$ is uniquely determined by limited angle data (Fourier slice and Paley-Wiener theorems). However, both inverse problems are very unstable (inversion is not Lipschitz continuous in any Sobolev norms, but one has conditional logarithmic stability).

\begin{definition}
A singularity at $(x_0, \xi_0)$ is called \emph{visible from $\mA$} if the line through $x_0$ in direction $\xi_0^{\perp}$ is in $\mA$.
\end{definition}

One has the following dichotomy:

\begin{itemize}
\item 
If $(x_0, \xi_0)$ is visible from $\mA$, then from the singularities of $Rf|_{\mA}$ one can determine for any $\alpha$ whether or not $(x_0, \xi_0) \in WF^{\alpha}(f)$. If $Rf|_{\mA}$ uniquely determines $f$, one expects the reconstruction of visible singularities to be stable.
\item 
If $(x_0, \xi_0)$ is not visible from $\mA$, then this singularity is smoothed out in the measurement $Rf|_{\mA}$. Even if $Rf|_{\mA}$ would determine $f$ uniquely, the inversion is not Lipschitz stable in any Sobolev norms.
\end{itemize}

\section{Gel'fand problem} \label{sec_gelfand}

Seismic imaging gives rise to various inverse problems related to determining interior properties, e.g.\ oil deposits or deep structure, of the Earth. Often this is done by using acoustic or elastic waves. We will consider the following problem, also known as the \emph{inverse boundary spectral problem} (see the monograph \cite{KKL}):

\begin{quote}
{\bf Gel'fand problem:} Is it possible to determine the interior structure of Earth by controlling acoustic waves and measuring vibrations at the surface?
\end{quote}

In seismic imaging one often tries to recover an unknown sound speed. However, in this presentation we consider the simpler case where the sound speed is constant (equal to one) and one attempts to recover an unknown potential $q \in C^{\infty}_c(\Omega)$ at each point $x \in \Omega$, where $\Omega$ is a ball in $\mR^n$.

Consider the free wave operator 
\[
\Box := \p_t^2 - \Delta.
\]
We assume that the medium is at rest at time $t=0$ and that we take measurements until time $T > 0$. If we prescribe the amplitude of the wave to be $f(x,t)$ on $\p \Omega \times (0,T)$, this leads to a solution $u$ of the wave equation 
\begin{equation} \label{wavedp}
\left\{ \begin{array}{rll}
(\Box + q) u &\!\!\!= 0 & \quad \text{in } \Omega \times (0,T), \\
u &\!\!\!= f & \quad \text{on } \partial \Omega \times (0,T), \\
u = \p_t u &\!\!\!= 0 & \quad \text{on } \{ t = 0 \}.
\end{array} \right.
\end{equation}
Given any $f \in C^{\infty}_c(\p \Omega \times (0,T))$, there is a unique solution $u \in C^{\infty}(\Omega \times (0,T))$ (see \cite[Theorem 7 in \S 7.2.3]{Evans}). We assume that we can measure the normal derivative $\p_{\nu} u|_{\p \Omega \times (0,T)}$, where $\p_{\nu} u(x,t) = \nabla_x u(x,t) \cdot \nu(x)$ and $\nu$ is the outer unit normal to $\p \Omega$. Doing such measurements for many different functions $f$, the ideal boundary measurements are encoded by the \emph{hyperbolic Dirichlet-to-Neumann map} (DN map for short) 
\[
\Lambda_q: C^{\infty}_c(\p \Omega \times (0,T)) \to C^{\infty}(\p \Omega \times (0,T)), \ \ \Lambda_q(f) = \p_{\nu} u|_{\p \Omega \times (0,T)}.
\]
The Gel'fand problem for this model amounts to recovering $q(x)$ from the knowledge of the map $\Lambda_q$. We will prove the following result due to \cite{RakeshSymes}.

\begin{theorem}[Recovering the X-ray transform] \label{thm_gelfand_xray}
Let $T > 0$ and assume that $q_1, q_2 \in C^{\infty}_c(\Omega)$. If 
\[
\Lambda_{q_1} = \Lambda_{q_2},
\]
then $q_1$ and $q_2$ satisfy 
\[
\int_{\gamma} q_1 \,ds = \int_{\gamma} q_2 \,ds
\]
whenever $\gamma$ is a maximal line segment in $\ol{\Omega}$ with length $< T$.
\end{theorem}

It is natural that the region where one can recover information depends on $T$. By finite propagation speed the map $\Lambda_q$ is unaffected if one changes $q$ outside the set \footnote{If $u$ and $\tilde{u}$ solve \eqref{wavedp} for potentials $q$ and $\tilde{q}$ with the same Dirichlet data $f$, and if $q = \tilde{q}$ in $U := \{x \in \Omega \,;\, \mathrm{dist}(x,\p \Omega) < T/2 \}$, then $w := u - \tilde{u}$ solves $(\Box + q)w = F$ where $F := -(q-\tilde{q})\tilde{u}$ vanishes in $U \times (0,T)$ and in $(\Omega \setminus U) \times (0,T/2)$. Moreover, $w = \p_t w = 0$ on $\{ t=0 \}$ and $w|_{\p \Omega \times (0,T)} = 0$. By finite speed of propagation $\p_{\nu} w|_{\p \Omega \times (0,T)} = 0$. This proves that $\Lambda_q = \Lambda_{\tilde{q}}$.}
\[
\{x \in \Omega \,;\, \mathrm{dist}(x,\p \Omega) < T/2 \}.
\]
For $T$ large enough, one can recover everything:

\begin{corollary}
If $T > \mathrm{diam}(\Omega)$, then $\Lambda_{q_1} = \Lambda_{q_2}$ implies $q_1 \equiv q_2$.
\end{corollary}
\begin{proof}
If $T > \mathrm{diam}(\Omega)$, then by Theorem \ref{thm_gelfand_xray} one has 
\[
\int_{\gamma} q_1 \,ds = \int_{\gamma} q_2 \,ds
\]
for any maximal line segment $\gamma$ in $\ol{\Omega}$. Thus $q_1$ and $q_2$ have the same X-ray transform in $\mR^n$. This transform is injective by Corollary \ref{corollary_radon_injectivity} when $n=2$. Tiling $\mR^n$ by two-planes gives injectivity when $n \geq 3$. Thus $q_1=q_2$.
\end{proof}

Theorem \ref{thm_gelfand_xray} could be proved based on the following facts, see e.g.\ \cite{StefanovYang}:
\begin{enumerate}
\item[1.]
The map $\Lambda_q$ is an FIO of order $1$ on $\p \Omega \times (0,T)$.
\item[2.]
The X-ray transform of $q$ can be read off from the symbol of $\Lambda_q$ (more precisely, from the principal symbol of $\Lambda_q - \Lambda_0$).
\end{enumerate}
We will give an elementary proof that is based on testing $\Lambda_q$ against highly oscillatory boundary data (compare with \eqref{principal_symbol_formula}).

The first step is an integral identity.

\begin{lemma}[Integral identity] \label{lemma_gelfand_integral_identity}
Assume that $q_1, q_2 \in C^{\infty}_c(\Omega)$. For any $f_1, f_2 \in C^{\infty}_c(\p \Omega \times (0,T))$, one has  
\[
((\Lambda_{q_1} - \Lambda_{q_2}) f_1, f_2)_{L^2(\p \Omega \times (0,T))} = \int_{\Omega} \int_0^T (q_1-q_2) u_1 \bar{u}_2 \,dt \,dx
\]
where $u_1$ solves \eqref{wavedp} with $q=q_1$ and $f=f_1$, and $u_2$ solves an analogous problem with vanishing Cauchy data on $\{Êt=T \}$:
\begin{equation} \label{wavedptwo}
\left\{ \begin{array}{rll}
(\Box + q_2) u_2 &\!\!\!= 0 & \quad \text{in } \Omega \times (0,T), \\
u_2 &\!\!\!= f_2 & \quad \text{on } \partial \Omega \times (0,T), \\
u_2 = \p_t u_2 &\!\!\!= 0 & \quad \text{on } \{ t = T \}.
\end{array} \right.
\end{equation}
\end{lemma}
\begin{proof}
We first compute the adjoint of the DN map: one has 
\[
( \Lambda_q f, g)_{L^2(\p \Omega \times (0,T))} = ( f, \Lambda_q^T g)_{L^2(\p \Omega \times (0,T))}
\]
where $\Lambda_q^T g = \p_{\nu} v|_{\p \Omega \times (0,T)}$ with $v$ solving $(\Box + q)v = 0$ so that $v|_{\p \Omega \times (0,T)} = g$ and $v = \p_t v = 0$ on $\{ t = T \}$. To prove this, we let $u$ be the solution of \eqref{wavedp} and integrate by parts:
\begin{align*}
( \Lambda_q f, g)_{L^2(\p \Omega \times (0,T))} &= \int_{\p \Omega} \int_0^T (\p_{\nu} u) \bar{v} \,dt \,dS \\
 &= \int_{\Omega} \int_0^T (\nabla u \cdot \nabla \bar{v} + (\Delta u) \bar{v}) \,dt \,dx \\
 &= \int_{\Omega} \int_0^T (\nabla u \cdot \nabla \bar{v} + (\p_t^2 u + qu) \bar{v}) \,dt \,dx \\
 &= \int_{\Omega} \int_0^T (\nabla u \cdot \nabla \bar{v} - \p_t u \p_t \bar{v} + qu \bar{v}) \,dt \,dx \\
 &= \int_{\Omega} \int_0^T (\nabla u \cdot \nabla \bar{v} + u (\ol{\p_t^2 v + qv}) ) \,dt \,dx \\
 &= \int_{\Omega} \int_0^T (\nabla u \cdot \nabla \bar{v} + u \Delta \ol{v}) \,dt \,dx \\
 &= \int_{\p \Omega} \int_0^T u \p_{\nu} \bar{v} \,dt \,dS \\
 &= ( f, \Lambda_q^T g)_{L^2(\p \Omega \times (0,T))}.
\end{align*}

Now, if $u_1$ and $u_2$ are as stated, the computation above gives 
\begin{align*}
( \Lambda_{q_1} f_1, f_2)_{L^2(\p \Omega \times (0,T))} 
 &= \int_{\Omega} \int_0^T (\nabla u_1 \cdot \nabla \bar{u}_2 - \p_t u_1 \p_t \bar{u}_2 + q_1 u_1 \bar{u}_2) \,dt \,dx \\
 \intertext{and} 
( \Lambda_{q_2} f_1, f_2)_{L^2(\p \Omega \times (0,T))} 
 &= ( f_1, \Lambda_{q_2}^T f_2)_{L^2(\p \Omega \times (0,T))} \\
 &= \int_{\Omega} \int_0^T (\nabla u_1 \cdot \nabla \bar{u}_2 - \p_t u_1 \p_t \bar{u}_2 + q_2 u_1 \bar{u}_2) \,dt \,dx. 
\end{align*}
The result follows by subtracting these two identities.
\end{proof}

The second step is to construct special solutions to the wave equation that concentrate near curves $s \mapsto (\gamma(s), s)$ where $\gamma$ is a line segment. These curves are projections to the $(x,t)$ variables of null bicharacteristic curves for $\Box$ (see Example \ref{example_bicharacteristic_wave}). Thus the following result is in line with Theorem \ref{thm_propagation_of_singularities} concerning propagation of singularities. The proof is based on a standard geometrical optics / WKB quasimode construction.

\begin{proposition}[Concentrating solutions] \label{prop_gelfand_concentrating_solutions}
Assume that $q \in C^{\infty}_c(\Omega)$, and let $\gamma: [\delta,L] \to \ol{\Omega}$ be a maximal line segment in $\ol{\Omega}$ with $0 < \delta < L < T$.
For any $\lambda \geq 1$ there is a solution $u = u_{\lambda}$ of $(\Box + q) u = 0$ in $\Omega \times (0,T)$ with $u = \p_t u = 0$ on $\{Êt = 0 \}$, such that for any $\psi \in C^{\infty}_c(\Omega \times [0,T])$ one has  
\begin{equation} \label{concentrating_solutions_first_limit}
\lim_{\lambda \to \infty} \int_{\Omega} \int_0^T \psi \abs{u}^2 \,dx \,dt  = \int_{\delta}^L \psi(\gamma(s), s) \,ds.
\end{equation}
Moreover, if $\tilde{q} \in C^{\infty}_c(\Omega)$, there is a solution $\tilde{u} = \tilde{u}_{\lambda}$ of $(\Box + \tilde{q})\tilde{u} = 0$ in $\Omega \times (0,T)$ with $\tilde{u} = \p_t \tilde{u} = 0$ on $\{Êt = T \}$, such that for any $\psi \in C^{\infty}_c(\Omega \times [0,T])$ one has  
\begin{equation} \label{concentrating_solutions_second_limit}
\lim_{\lambda \to \infty} \int_{\Omega} \int_0^T \psi u \ol{\tilde{u}} \,dt \,dx  = \int_{\delta}^L \psi(\gamma(s), s) \,ds.
\end{equation}
\end{proposition}

At this point it is easy to prove the main result:

\begin{proof}[Proof of Theorem \ref{thm_gelfand_xray}]
Using the assumption $\Lambda_{q_1} = \Lambda_{q_2}$ and Lemma \ref{lemma_gelfand_integral_identity}, we have 
\begin{equation} \label{gelfand_orthogonality_relation}
\int_{\Omega} \int_0^T (q_1-q_2) u_1 \ol{u}_2 \,dt \,dx = 0
\end{equation}
for any solutions $u_j$ of $(\Box + q_j) u_j = 0$ in $\Omega \times (0,T)$ so that $u_1 = \p_t u_1 = 0$ on $\{ t = 0 \}$, and $u_2 = \p_t u_2 = 0$ on $\{ t = T \}$.

Let $\gamma: [\delta,L] \to \ol{\Omega}$ be a maximal unit speed line segment in $\ol{\Omega}$ with $L < T$, and let $u_1 = u_{1,\lambda}$ be the solution constructed in Proposition \ref{prop_gelfand_concentrating_solutions} for the potential $q_1$ with $u_1 = \p_t u_1 = 0$ on $\{ t = 0 \}$. Moreover, let $u_2 = u_{2,\lambda}$ be the solution constructed in the end of Proposition \ref{prop_gelfand_concentrating_solutions} for the potential $q_2$ with $u_1 = \p_t u_1 = 0$ on $\{ t = T \}$. Taking the limit as $\lambda \to \infty$ in \eqref{gelfand_orthogonality_relation} and using \eqref{concentrating_solutions_second_limit} with $\psi(x,t) = (q_1-q_2)(x)$, we obtain that 
\[
\int_{\delta}^L (q_1-q_2)(\gamma(s)) \,ds = 0.
\]
Thus the integrals of $q_1$ and $q_2$ over maximal line segments of length $< T$ in $\ol{\Omega}$ are the same.
\end{proof}

\begin{proof}[Proof of Proposition \ref{prop_gelfand_concentrating_solutions}]
Let $\gamma: [\delta,L] \to \ol{\Omega}$ be a maximal unit speed line segment in $\ol{\Omega}$ with $L < T$, and let $\eta: \mR \to \mR^n$ be the unit speed line so that $\eta(s) = \gamma(s)$ for $s \in [\delta, L]$. Write $x_0 := \eta(0)$ and $\xi_0 := \dot{\eta}(0)$, so that $x_0Ê\notin \ol{\Omega}$ and $\gamma(s) = x_0 + (s+\delta)\xi_0$. After a translation and rotation, we may assume that $x_0 = 0$ and $\xi_0 = e_n$.

We first construct an approximate solution $v = v_{\lambda}$ for the operator $\Box + q$, having the form 
\[
v(x,t) = e^{i\lambda \varphi(x,t)} a(x,t)
\]
where $\varphi$ is a real phase function, and $a$ is an amplitude supported near the curve $s \mapsto (\eta(s), s)$. Note that 
\begin{align*}
\p_t(e^{i\lambda \varphi} u) &= e^{i\lambda \varphi}(\p_t + i \lambda \p_t \varphi) u, \\
\p_t^2(e^{i\lambda \varphi} u) &= e^{i\lambda \varphi}(\p_t + i \lambda \p_t \varphi)^2 u.
\end{align*}
Using a similar expression for $\p_{x_j}^2$, we compute 
\begin{align}
(\Box+q)(e^{i\lambda \varphi} a) &= e^{i\lambda \varphi}( (\p_t + i\lambda \p_t \varphi)^2 - (\nabla_x + i \lambda \nabla_x \varphi)^2 + q) a \notag \\
 &= e^{i\lambda \varphi} \big[ \lambda^2 \left[ \abs{\nabla_x \varphi}^2 - (\p_t \varphi)^2 \right] a \notag \\
 & \qquad + i \lambda \left[2 \p_t \varphi \p_t a - 2 \nabla_x \varphi \cdot \nabla_x a + (\Box \varphi)a \right] + (\Box + q)a \big]. \label{wave_equation_geometric_optics}
\end{align}

We would like to have $(\Box+q)(e^{i\lambda \varphi} a) = O(\lambda^{-1})$. To this end, we first choose $\varphi$ so that the $\lambda^2$ term in \eqref{wave_equation_geometric_optics} vanishes. This will be true if $\varphi$ solves the \emph{eikonal equation} 
\[
\abs{\nabla_x \varphi}^2 - (\p_t \varphi)^2 = 0.
\]
There are many possible solutions, but we make the simple choice 
\[
\varphi(x,t) := t - x_n.
\]
With this choice, \eqref{wave_equation_geometric_optics} becomes 
\begin{equation} \label{wave_equation_geometric_optics_two}
(\Box+q)(e^{i\lambda \varphi} a) = e^{i\lambda \varphi} \left[ i\lambda (La) + (\Box+q)a \right]
\end{equation}
where $L$ is the constant vector field 
\[
L := 2 (\p_t + \p_{x_n}).
\]
It is convenient to consider new coordinates $(x',z,w)$ in $\mR^{n+1}$, where 
\begin{equation} \label{new_z_w_coordinates}
z = \frac{t+x_n}{2}, \qquad w = \frac{t-x_n}{2}.
\end{equation}
Then $L$ corresponds to $2 \p_z$ in the sense that 
\[
LF(x,t) = 2 \p_z \breve{F}(x',\frac{t+x_n}{2},\frac{t-x_n}{2})
\]
where $\breve{F}$ corresponds to $F$ in the new coordinates:
\[
\breve{F}(x',z,w) := F(x',z-w,z+w).
\]

We next look for the amplitude $a$ in the form 
\[
a = a_0 + \lambda^{-1} a_{-1}.
\]
Inserting this to \eqref{wave_equation_geometric_optics} and equating like powers of $\lambda$, we get 
\begin{equation} \label{wave_equation_geometric_optics_three}
(\Box+q)(e^{i\lambda \varphi} a) = e^{i\lambda \varphi} \left[ i\lambda (La_0) + \left[ i L a_{-1} + (\Box+q)a_0 \right] + \lambda^{-1} (\Box+q)a_{-1} \right].
\end{equation}
We would like the last expression to be $O(\lambda^{-1})$. This will hold if $a_0$ and $a_{-1}$ satisfy the \emph{transport equations} 
\begin{align} \label{wave_transport_equations}
\left\{ \begin{array}{rl}
L a_0 &\!\!\!= 0, \\[3pt]
La_{-1} &\!\!\!= i(\Box+q)a_0.
\end{array} \right.
\end{align}
Let $\chi \in C^{\infty}_c(\mR^n)$ be supported near $0$, and choose 
\[
\breve{a}_0(x',z,w) := \chi(x',w).
\]
We will later choose $\chi$ to depend on $\lambda$. Next we choose 
\[
\breve{a}_{-1}(x',z,w) := -\frac{1}{2i} \int_0^z ((\Box + q)a_0){\,\breve{\rule{0pt}{6pt}}\,}(x',s,w) \,ds.
\]
These functions satisfy \eqref{wave_transport_equations}, and they vanish unless $w$ is small (i.e.\ $x_n$ is close to $t$). Then \eqref{wave_equation_geometric_optics_three} becomes 
\[
(\Box+q)(e^{i\lambda \varphi} a) = F_{\lambda}
\]
where 
\[
F_{\lambda} := \lambda^{-1} e^{i\lambda \varphi} (\Box+q)a_{-1}.
\]
Using the Cauchy-Schwartz inequality, one can check that 
\begin{align*}
\norm{F_{\lambda}}_{L^{\infty}(\Omega \times (0,T))} &\leq \lambda^{-1} \norm{(\Box + q) a_{-1}}_{L^{\infty}(\Omega \times (0,T))} \\
 &\lesssim \lambda^{-1} \norm{\chi}_{W^{4,\infty}(\mR^n)}
\end{align*}
uniformly over $\lambda \geq 1$. This concludes the construction of the approximate solution $v = e^{i\lambda \varphi} a$.

We next find an exact solution $u = u_{\lambda}$ of \eqref{wavedp} having the form 
\[
u = v + r
\]
where $r$ is a correction term. Note that for $t$ close to $0$, $v(\,\cdot\,,t)$ is supported near $x_0 \notin \ol{\Omega}$ and hence $v = \p_t v = 0$ on $\{t=0\}$. Note also that $(\Box + q)v = F_{\lambda}$. Thus $u$ will solve \eqref{wavedp} for $f = v|_{\p \Omega \times (0,T)}$ if $r$ solves 
\begin{equation} \label{wavedp_correction_term}
\left\{ \begin{array}{rll}
(\Box + q) r &\!\!\!= -F_{\lambda} & \quad \text{in } \Omega \times (0,T), \\
r &\!\!\!= 0 & \quad \text{on } \partial \Omega \times (0,T), \\
r = \p_t r &\!\!\!= 0 & \quad \text{on } \{ t = 0 \}.
\end{array} \right.
\end{equation}
By the wellposedness of this problem \cite[Theorem 5 in \S 7.2.3]{Evans}, there is a unique solution $r$ with 
\[
\norm{r}_{L^{\infty}((0,T) ; H^1(\Omega))} \lesssim \norm{F_{\lambda}}_{L^2((0,T) ; L^2(\Omega))} \lesssim \lambda^{-1} \norm{\chi}_{W^{4,\infty}}.
\]

We now fix the choice of $\chi$ so that \eqref{concentrating_solutions_first_limit} will hold. Let $\zeta \in C^{\infty}_c(\mR^n)$ satisfy $\zeta = 1$ near $0$ and $\norm{\zeta}_{L^2(\mR^n)} = 1$, and choose 
\[
\chi(y) := \eps^{-n/2} \zeta(y/\eps)
\]
where 
\[
\eps = \eps(\lambda) = \lambda^{-\frac{1}{n+8}}.
\]
With this choice 
\[
\norm{\chi}_{L^2(\mR^n)} = 1, \qquad \norm{\chi}_{W^{4,\infty}(\mR^n)} \lesssim \eps^{-n/2-4} \lesssim \lambda^{1/2}.
\]
It follows that 
\[
\norm{v}_{L^2(\Omega \times (0,T))} \lesssim 1, \qquad \norm{r}_{L^2(\Omega \times (0,T))} \lesssim \lambda^{-1/2}.
\]
Since $u = v + r$, the integral in \eqref{concentrating_solutions_first_limit} has the form 
\begin{align*}
\int_{\Omega} \int_0^T \psi \abs{u}^2 \,dx \,dt &= \int_{\Omega} \int_0^T \psi \abs{v}^2 \,dx \,dt + O(\lambda^{-1/2}) \\
 &= \int_{\Omega} \int_0^T \psi \abs{a_0}^2 \,dx \,dt + O(\lambda^{-1/2}).
\end{align*}
Using that $\psi \abs{a_0}^2$ is compactly supported in $\Omega \times (0,T)$, we have 
\begin{align*}
 &\int_{\Omega} \int_0^T \psi \abs{u}^2 \,dx \,dt = \int_{\mR^{n+1}} \psi(x,t) \eps^{-n} \zeta(\frac{x'}{\eps}, \frac{t-x_n}{2\eps})^2 \,dx \,dt + O(\lambda^{-1/2}) \\
 &= \int_{\mR^{n+1}} \psi(x',z-w,z+w) \eps^{-n} \zeta(x'/\eps, w/\eps)^2 \,dx' \,dz \,dw + O(\lambda^{-1/2})
\end{align*}
by changing variables as in \eqref{new_z_w_coordinates}. Finally, changing $x'$ to $\eps x'$ and $w$ to $\eps w$ and letting $\lambda \to \infty$ (so $\eps \to 0$) yields 
\begin{align*}
\lim_{\lambda \to \infty} \int_{\Omega} \int_0^T \psi \abs{u}^2 \,dx \,dt &= \int_{\mR^{n+1}} \psi(0',z,z) \zeta(x',w)^2 \,dx' \,dz \,dw \\
 &= \int_{-\infty}^{\infty} \psi(0',z,z) \,dz =  \int_{\delta}^L \psi(x_0 + s e_n, s) \,ds
\end{align*}
by the normalization $\norm{\zeta}_{L^2(\mR^n)} = 1$ and the fact that $\psi \in C^{\infty}_c(\Omega \times [0,T])$. This proves \eqref{concentrating_solutions_first_limit}.

It remains to prove \eqref{concentrating_solutions_second_limit}. Since $\eta(T) \notin \ol{\Omega}$, we have $v = \p_t v = 0$ on $\{ t=T \}$, and we may alternatively arrange that $r$ solves \eqref{wavedp_correction_term} with $r = \p_t r = 0$ on $\{ t=T \}$ instead of $\{ t=0 \}$. We can do such a construction for the potential $\tilde{q}$ instead of $q$. Since $\varphi$ and $a_0$ are independent of $q$, the same argument as above proves \eqref{concentrating_solutions_second_limit}.
\end{proof}

\section{Calder\'on problem: boundary determination} \label{sec_calderon_boundary}

Electrical Impedance Tomography (EIT) is an imaging method with potential applications in medical imaging and nondestructive testing. The method is based on the following important inverse problem.

\begin{quote}
{\bf Calder\'on problem:} Is it possible to determine the electrical conductivity of a medium by making voltage and current measurements on its boundary?
\end{quote}

The treatment in this section follows \cite{FSU}.

Let us begin by recalling the mathematical model of EIT. The purpose is to determine the electrical conductivity $\gamma(x)$ at each point $x \in \Omega$, where $\Omega \subset \mR^n$ represents the body which is imaged (in practice $n=3$). We assume that $\Omega \subset \mR^n$ is a bounded open set with $C^{\infty}$ boundary, and that $\gamma \in C^{\infty}(\closure{\Omega})$ is positive.

Under the assumption of no sources or sinks of current in $\Omega$, a voltage potential $f$ at the boundary $\partial \Omega$ induces a voltage potential $u$ in $\Omega$, which solves the Dirichlet problem for the conductivity equation, 
\begin{equation} \label{conductivitydp}
\left\{ \begin{array}{rll}
\nabla \cdot \gamma \nabla u &\!\!\!= 0 & \quad \text{in } \Omega, \\
u &\!\!\!= f & \quad \text{on } \partial \Omega.
\end{array} \right.
\end{equation}
Since $\gamma \in C^{\infty}(\closure{\Omega})$ is positive, the equation is uniformly elliptic, and there is a unique solution $u \in C^{\infty}(\ol{\Omega})$ for any boundary value $f \in C^{\infty}(\partial \Omega)$. One can define the Dirichlet-to-Neumann map (DN map) as 
\begin{equation*}
\Lambda_{\gamma}: C^{\infty}(\partial \Omega) \to C^{\infty}(\partial \Omega),  \ \ f \mapsto \gamma \p_{\nu} u |_{\partial \Omega}.
\end{equation*}
Here $\nu$ is the outer unit normal to $\p \Omega$ and $\p_{\nu} u|_{\p \Omega} = \nabla u \cdot \nu|_{\p \Omega}$ is the normal derivative of $u$. Physically, $\Lambda_{\gamma} f$ is the current flowing through the boundary.

The Calder\'on problem (also called the inverse conductivity problem) is to determine the conductivity function $\gamma$ from the knowledge of the map $\Lambda_{\gamma}$. That is, if the measured current $\Lambda_{\gamma} f$ is known for all boundary voltages $f \in C^{\infty}(\p \Omega)$, one would like to determine the conductivity $\gamma$.

We will prove the following theorem.

\begin{theorem}[Boundary determination] \label{thm_calderon_boundary_determination}
Let $\gamma_1, \gamma_2 \in C^{\infty}(\ol{\Omega})$ be positive. If 
\[
\Lambda_{\gamma_1} = \Lambda_{\gamma_2},
\]
then the Taylor series of $\gamma_1$ and $\gamma_2$ coincide at any point of $\p \Omega$.
\end{theorem}

This result was proved in \cite{KohnVogelius}, and it in particular implies that any real-analytic conductivity is uniquely determined by the DN map. The argument extends to piecewise real-analytic conductivities. A different proof was given in \cite{SylvesterUhlmann_boundary}, based on two facts:

\begin{enumerate}
\item[1.]
The DN map $\Lambda_{\gamma}$ is an elliptic $\Psi$DO of order $1$ on $\p \Omega$.
\item[2.] 
The Taylor series of $\gamma$ at a boundary point can be read off from the symbol of $\Lambda_{\gamma}$ computed in suitable coordinates. The symbol of $\Lambda_{\gamma}$ can be computed by testing against highly oscillatory boundary data (compare with \eqref{principal_symbol_formula}).
\end{enumerate}

\begin{remark}
The above argument is based on studying the singularities of the integral kernel of the DN map, and it only determines the Taylor series of the conductivity at the boundary. The values of the conductivity in the interior are encoded in the $C^{\infty}$ part of the kernel, and different methods (based on \emph{complex geometrical optics solutions}) are required for interior determination.
\end{remark}

Let us start with a simple example:

\begin{example}[DN map in half space is a $\Psi$DO]
Let $\Omega = \mR^n_+ = \{Êx_n > 0 \}$, so $\p \Omega = \mR^{n-1} = \{Êx_n = 0 \}$. We wish to compute the DN map for the Laplace equation (i.e.\ $\gamma \equiv 1$) in $\Omega$. Consider 
\[
\left\{ \begin{array}{rll}
\Delta u &\!\!\!= 0 & \quad \text{in } \mR^n_+, \\
u &\!\!\!= f & \quad \text{on } \{Êx_n = 0 \}.
\end{array} \right.
\]
Writing $x = (x',x_n)$ and taking Fourier transforms in $x'$ gives 
\[
\left\{ \begin{array}{rll}
(\p_n^2 - \abs{\xi'}^2) \hat{u}(\xi',x_n) &\!\!\!= 0 & \quad \text{in } \mR^n_+, \\
\hat{u}(\xi',0) &\!\!\!= \hat{f}(\xi'). &
\end{array} \right.
\]
Solving this ODE for fixed $\xi'$ and choosing the solution that decays for $x_n > 0$ gives 
\begin{align*}
 &\hat{u}(\xi',x_n) = e^{-x_n \abs{\xi'}} \hat{f}(\xi') \\
 &\implies u(x',x_n) = \mF_{\xi'}^{-1} \left\{ e^{-x_n \abs{\xi'}} \hat{f}(\xi') \right\}.
\end{align*}
We may now compute the DN map:
\[
\Lambda_1 f = -\p_n u|_{x_n = 0} =  \mF_{\xi'}^{-1} \left\{ \abs{\xi'} \hat{f}(\xi') \right\}.
\]
Thus the DN map on the boundary $\p \Omega = \mR^{n-1}$ is just $\Lambda_1 = \abs{D_{x'}}$ corresponding to the Fourier multiplier $\abs{\xi'}$. This shows that at least in this simple case, the DN map is an elliptic $\Psi$DO of order $1$.
\end{example}

We will now prove Theorem \ref{thm_calderon_boundary_determination} by an argument that avoids showing that the DN map is a $\Psi$DO, but is rather based on directly testing the DN map against oscillatory boundary data. The first step is a basic integral identity (sometimes called Alessandrini identity) for the DN map.

\begin{lemma}[Integral identity] \label{lemma_calderon_integral_identity}
Let $\gamma_1, \gamma_2 \in C^{\infty}(\ol{\Omega})$. If $f_1, f_2 \in C^{\infty}(\p \Omega)$, then 
\[
((\Lambda_{\gamma_1} - \Lambda_{\gamma_2}) f_1, f_2)_{L^2(\p \Omega)} = \int_{\Omega} (\gamma_1 - \gamma_2) \nabla u_1 \cdot \nabla \bar{u}_2 \,dx
\]
where $u_j \in C^{\infty}(\ol{\Omega})$ solves $\mdiv(\gamma_j \nabla u_j) = 0$ in $\Omega$ with $u_j|_{\p \Omega} = f_j$.
\end{lemma}
\begin{proof}
We first observe that the DN map is symmetric: if $\gamma \in C^{\infty}(\ol{\Omega})$ is positive and if $u_f$ solves $\nabla \cdot (\gamma \nabla u_f) = 0$ in $\Omega$ with $u_f|_{\p \Omega} = f$, then an integration by parts shows that 
\begin{align*}
(\Lambda_{\gamma} f, g)_{L^2(\p \Omega)} &= \int_{\p \Omega} (\gamma \p_{\nu} u_f) \ol{u}_g \,dS = \int_{\Omega} \gamma \nabla u_f \cdot \nabla \ol{u}_g \,dx \\
 &= \int_{\p \Omega} u_f (\ol{\gamma \p_{\nu} u_g}) \,dS = (f, \Lambda_{\gamma} g)_{L^2(\p \Omega)}.
\end{align*}
Thus 
\begin{align*}
(\Lambda_{\gamma_1} f_1, f_2)_{L^2(\p \Omega)} &= \int_{\Omega} \gamma_1 \nabla u_1 \cdot \nabla \ol{u}_2 \,dx, \\
(\Lambda_{\gamma_2} f_1, f_2)_{L^2(\p \Omega)} &= (f_1, \Lambda_{\gamma_2} f_2)_{L^2(\p \Omega)} = \int_{\Omega} \gamma_2 \nabla u_1 \cdot \nabla \ol{u}_2 \,dx.
\end{align*}
The result follows by subtracting the above two identities.
\end{proof}

Next we show that if $x_0$ is a boundary point, there is an approximate solution of the conductivity equation that concentrates near $x_0$, has highly oscillatory boundary data, and decays exponentially in the interior. As a simple example, the solution of 
\[
\left\{ \begin{array}{rll}
\Delta u &\!\!\!= 0 & \quad \text{in } \mR^n_+, \\
u(x',0) &\!\!\!= e^{i\lambda x' \cdot \xi'} &
\end{array} \right.
\]
that decays for $x_n > 0$ is given by $u = e^{-\lambda x_n} e^{i\lambda x' \cdot \xi'}$, which concentrates near $\{Êx_n = 0 \}$ and decays exponentially when $x_n > 0$ if $\lambda$ is large. Roughly, this means that the solution of a Laplace type equation with highly oscillatory boundary data concentrates near the boundary.

\begin{proposition} \label{prop_calderon_oscillating_solutions}
(Concentrating approximate solutions) Let $\gamma \in C^{\infty}(\ol{\Omega})$ be positive, let $x_0 \in \p \Omega$, let $\xi_0$ be a unit tangent vector to $\p \Omega$ at $x_0$, and let $\chi \in C^{\infty}_c(\p \Omega)$ be supported near $x_0$. Let also $N \geq 1$. For any $\lambda \geq 1$ there exists $v = v_{\lambda} \in C^{\infty}(\ol{\Omega})$ having the form 
\[
v = \lambda^{-1/2} e^{i\lambda \Phi} a
\]
such that 
\begin{gather*}
\nabla \Phi(x_0) = \xi_0 -i \nu(x_0), \\
\text{$a$ is supported near $x_0$ with $a|_{\p \Omega} = \chi$},
\end{gather*}
and as $\lambda \to \infty$ 
\[
\norm{v}_{H^1(\Omega)} \sim 1, \qquad \norm{\mdiv(\gamma \nabla v)}_{L^2(\Omega)} = O(\lambda^{-N}).
\]
Moreover, if $\tilde{\gamma} \in C^{\infty}(\ol{\Omega})$ is positive and $\tilde{v} = \tilde{v}_{\lambda}$ is the corresponding approximate solution constructed for $\tilde{\gamma}$, then for any $f \in C(\ol{\Omega})$ and $k \geq 0$ one has 
%$\Phi$ and each $a_{-j} \in C^{\infty}(\ol{\Omega})$ are independent of $\gamma$ and $\lambda$, and for any $f \in C(\ol{\Omega})$ supported near $x_0$ one has 
\begin{equation} \label{boundary_determination_limit}
\lim_{\lambda \to \infty} \lambda^k \int_{\Omega} \mathrm{dist}(x,\p \Omega)^k f \nabla v \cdot \ol{\nabla \tilde{v}} \,dx = c_k \int_{\p \Omega} f \abs{\chi}^2 \,dS.
\end{equation}
for some $c_k \neq 0$.
\end{proposition}

We can now give the proof of the boundary determination result.

\begin{proof}[Proof of Theorem \ref{thm_calderon_boundary_determination}]
Using the assumption that $\Lambda_{\gamma_1} = \Lambda_{\gamma_2}$ together with the integral identity in Lemma \ref{lemma_calderon_integral_identity}, we have that 
\begin{equation} \label{calderon_integral_identity_vanishing}
\int_{\Omega} (\gamma_1 - \gamma_2) \nabla u_1 \cdot \nabla \bar{u}_2 \,dx = 0
\end{equation}
whenever $u_j$ solves $\mdiv(\gamma_j \nabla u_j) = 0$ in $\Omega$.

Let $x_0 \in \p \Omega$, let $\xi_0$ be a unit tangent vector to $\p \Omega$ at $x_0$, and let $\chi \in C^{\infty}_c(\p \Omega)$ satisfy $\chi = 1$ near $x_0$. We use Proposition \ref{prop_calderon_oscillating_solutions} to construct functions 
\[
v_j = v_{j,\lambda} = \lambda^{-1/2} e^{i\lambda \Phi} a_j
\]
so that 
\begin{equation} \label{vj_estimates}
\norm{v_j}_{H^1(\Omega)} \sim 1, \qquad \norm{\mdiv(\gamma_j \nabla v_j)}_{L^2(\Omega)} = O(\lambda^{-N}).
\end{equation}
We obtain exact solutions $u_j$ of $\mdiv(\gamma_j \nabla u_j) = 0$ by setting 
\[
u_j := v_j + r_j,
\]
where the correction terms $r_j$ are the unique solutions of 
\[
\mdiv(\gamma_j \nabla r_j) = -\mdiv(\gamma_j \nabla v_j) \text{ in $\Omega$}, \qquad r_j|_{\p \Omega} = 0.
\]
By standard energy estimates \cite[Section 6.2]{Evans} and by \eqref{vj_estimates}, the solutions $r_j$ satisfy 
\begin{equation} \label{calderon_correction_term_estimate}
\norm{r_j}_{H^1(\Omega)} \lesssim \norm{\mdiv(\gamma_j \nabla v_j)}_{H^{-1}(\Omega)} = O(\lambda^{-N}).
\end{equation}

We now insert the solutions $u_j = v_j + r_j$ into \eqref{calderon_integral_identity_vanishing}. Using \eqref{calderon_correction_term_estimate} and \eqref{vj_estimates}, it follows that 
\begin{equation} \label{gamma_difference_approximate}
\int_{\Omega} (\gamma_1 - \gamma_2) \nabla v_1 \cdot \nabla \bar{v}_2 \,dx = O(\lambda^{-N})
\end{equation}
as $\lambda \to \infty$. Letting $\lambda \to \infty$, the formula \eqref{boundary_determination_limit} yields 
\[
\int_{\p \Omega} (\gamma_1 - \gamma_2) \abs{\chi}^2 \,dS = 0.
\]
In particular, $\gamma_1(x_0) = \gamma_2(x_0)$.

We will prove by induction that 
\begin{equation} \label{pnuj_induction_claim}
\p_{\nu}^j \gamma_1|_{\p \Omega} = \p_{\nu}^j \gamma_2|_{\p \Omega} \text{ near $x_0$ for any $j \geq 0$.}
\end{equation}
The case $j=0$ was proved above (here we may vary $x_0$ slightly). We make the induction hypothesis that \eqref{pnuj_induction_claim} holds for $j \leq k-1$. Let $(x', x_n)$ be boundary normal coordinates so that $x_0$ corresponds to $0$, and $\p \Omega$ near $x_0$ corresponds to $\{ x_n = 0 \}$. The induction hypothesis states that 
\[
\p_n^j \gamma_1(x',0) = \p_n^j \gamma_2(x',0), \qquad j \leq k-1.
\]
Considering the Taylor expansion of $(\gamma_1-\gamma_2)(x',x_n)$ with respect to $x_n$ gives that 
\[
(\gamma_1-\gamma_2)(x',x_n) = x_n^k f(x',x_n) \text{ near $0$ in $\{ x_n \geq 0 \}$}
\]
for some smooth function $f$ with $f(x',0) = \frac{\p_n^k (\gamma_1-\gamma_2)(x',0)}{k!}$. Inserting this formula in \eqref{gamma_difference_approximate}, we obtain that 
\[
\lambda^k \int_{\Omega} x_n^k f \nabla v_1 \cdot \nabla \bar{v}_2 \,dx = O(\lambda^{k-N}).
\]
Now $x_n = \mathrm{dist}(x,\p \Omega)$ in boundary normal coordinates. Assuming that $N$ was chosen larger than $k$, we may take the limit as $\lambda \to \infty$ and use \eqref{boundary_determination_limit} to obtain that 
\[
\int_{\p \Omega} f(x',0) \abs{\chi(x',0)}^2 \,dS(x') = 0.
\]
This shows that $\p_n^k (\gamma_1-\gamma_2)(x',0) = 0$ for $x'$ near $0$, which concludes the induction.
\end{proof}

It remains to prove Proposition \ref{prop_calderon_oscillating_solutions}, which constructs approximate solutions (also called \emph{quasimodes}) concentrating near a boundary point. This is a typical geometrical optics / WKB type construction for quasimodes with complex phase. The proof is elementary, although a bit long. The argument is simplified slightly by using the Borel summation lemma, which is used frequently in microlocal analysis in various different forms.

\begin{lemma}[Borel summation, {{\cite[Theorem 1.2.6]{Hormander}}}] \label{lemma_borel_summation}
Let $f_j \in C^{\infty}_c(\mR^{n-1})$ for $j = 0, 1, 2, \ldots$. There exists $f \in C^{\infty}_c(\mR^n)$ such that 
\[
\p_n^j f(x',0) = f_j(x'), \qquad j=0,1,2,\ldots.
\]
\end{lemma}

\begin{proof}[Proof of Proposition \ref{prop_calderon_oscillating_solutions}]
We will first carry out the proof in the case where $x_0 = 0$ and $\p \Omega$ is flat near $0$, i.e.\ $\Omega \cap B(0,r) = \{Êx_n > 0 \} \cap B(0,r)$ for some $r > 0$ (the general case will be considered in the end of the proof). We also assume $\xi_0 = (\xi_0',0)$ where $\abs{\xi_0'} = 1$.

We look for $v$ in the form 
\[
v = e^{i\lambda \Phi} b.
\]
Write $Pu = D \cdot (\gamma D u) = \gamma D^2 u + D\gamma \cdot Du$. The principal symbol of $P$ is 
\begin{equation} \label{ptwo_formula_simple}
p_2(x,\xi) := \gamma(x) \xi \cdot \xi.
\end{equation}
Since $e^{-i\lambda \Phi}D_j(e^{i\lambda \Phi} b) = (D_j + \lambda \p_j \Phi)b$, we compute 
\begin{align}
P(e^{i\lambda \Phi} b) &= e^{i\lambda \Phi}(D + \lambda \nabla \Phi) \cdot (\gamma (D + \lambda \nabla \Phi) b) \notag \\
 &= e^{i\lambda \Phi}\left[ \lambda^2 p_2(x,\nabla \Phi) b + \lambda \frac{1}{i} \left[ \underbrace{2 \gamma \nabla \Phi \cdot \nabla b + \nabla \cdot (\gamma \nabla \Phi) b}_{=: Lb} \right] + Pb \right] \label{p_eilambdaphi_formula}
\end{align}

We want to choose $\Phi$ and $b$ so that $P(e^{i\lambda \Phi} b) = O_{L^2(\Omega)}(\lambda^{-N})$. Looking at the $\lambda^2$ term in \eqref{p_eilambdaphi_formula}, we first choose $\Phi$ so that 
\begin{equation} \label{ptwo_phi_equation}
p_2(x,\nabla \Phi) = 0 \text{ in $\Omega$}.
\end{equation}
We additionally want that $\Phi(x',0) = x' \cdot \xi_0'$ and $\p_n \Phi(x',0) = i$ (this will imply that $\nabla \Phi(0) = \xi_0 + i e_n$). In fact, using \eqref{ptwo_formula_simple} we can just choose 
\[
\Phi(x',x_n) := x' \cdot \xi_0' + i x_n
\]
and then $p_2(x,\nabla \Phi) = \gamma(\xi_0+ie_n) \cdot (\xi_0 + ie_n) \equiv 0$ in $\Omega$.

We next look for $b$ in the form 
\[
b = \sum_{j=0}^N \lambda^{-j} b_{-j}.
\]
Since $p_2(x,\nabla \Phi) \equiv 0$, \eqref{p_eilambdaphi_formula} implies that 
\begin{align}
P(e^{i\lambda \Phi} b) &= e^{i\lambda \Phi} \Big[ \lambda [ \frac{1}{i} Lb_0 ] + [ \frac{1}{i} Lb_{-1} + Pb_0] + \lambda^{-1} [ \frac{1}{i} Lb_{-2} + Pb_{-1}] + \ldots \notag \\
 &\hspace{30pt} + \lambda^{-(N-1)} [ \frac{1}{i} Lb_{-N} + Pb_{-(N-1)}] + \lambda^{-N} Pb_{-N }\Big] \label{p_eilambdaphi_formulatwo}.
\end{align}
We will choose the functions $b_{-j}$ so that 
\begin{align} \label{lbone} 
\left\{ \begin{array}{rl}
Lb_0 &\!\!\!= 0 \text{ to infinite order at $\{Êx_n = 0 \}$}, \\[5pt]
\frac{1}{i} Lb_{-1} + P b_0 &\!\!\!= 0 \text{ to infinite order at $\{Êx_n = 0 \}$}, \\[5pt]
&\!\vdots \\[5pt]
\frac{1}{i} Lb_{-N} + P b_{-(N-1)} &\!\!\!= 0 \text{ to infinite order at $\{Êx_n = 0 \}$}. 
\end{array} \right.
\end{align}
We will additionally arrange that 
\begin{align} \label{bzeroone}
\left\{ \begin{array}{rl}
b_0(x',0) &\!\!\!= \chi(x'), \\[5pt]
b_{-j}(x',0) &\!\!\!= 0 \text{ for $1 \leq j \leq N$},
\end{array} \right.
\end{align}
and that each $b_{-j}$ is compactly supported so that 
\begin{equation} \label{bsupport}
\mathrm{supp}(b_{-j}) \subset Q_{\eps} := \{Ê\abs{x'} < \eps, \ 0 \leq x_n < \eps \}
\end{equation}
for some fixed $\eps > 0$.

To find $b_0$, we prescribe $b_0(x',0), \p_n b_0(x',0), \p_n^2 b_0(x',0)$, $\ldots$ successively and use the Borel summation lemma to construct $b_0$ with this Taylor series at $\{Êx_n = 0 \}$. We first set $b_0(x',0) = \chi(x')$. Writing $\eta := \nabla \cdot (\gamma \nabla \Phi)$, we observe that 
\[
Lb_0|_{x_n=0} = 2 \gamma (\xi_0' \cdot \nabla_{x'} b_0 + i\p_n b_0) + \eta b_0|_{x_n=0}. 
\]
Thus, in order to have $Lb_0|_{x_n=0} = 0$ we must have 
\[
\p_n b(x',0) = -\frac{1}{2i \gamma(x',0)} \left[ 2 \gamma(x',0) \xi_0' \cdot \nabla_{x'} b_0 + \eta b_0 \right] \Big|_{x_n=0}.
\]
We prescribe $\p_n b(x',0)$ to have the above value (which depends on the already prescribed quantity $b(x',0)$). Next we compute 
\[
\p_n (Lb_0)|_{x_n=0} = 2 \gamma i \p_n^2 b_0 + Q(x', b_0(x',0), \p_n b_0(x',0))
\]
where $Q$ depends on the already prescribed quantities $b_0(x',0)$ and $\p_n b_0(x',0)$. We thus set 
\[
\p_n^2 b_0(x',0) = -\frac{1}{2i\gamma(x',0)} Q(x', b_0(x',0), \p_n b_0(x',0)),
\]
which ensures that $\p_n (Lb_0)|_{x_n=0} = 0$. Continuing in this way and using Borel summation, we obtain a function $b_0$ so that $Lb_0 = 0$ to infinite order at $\{ x_n = 0 \}$. The other equations in \eqref{lbone} are solved in a similar way, which gives the required functions $b_{-1}, \ldots, b_{-N}$. In the construction, we may arrange so that \eqref{bzeroone} and \eqref{bsupport} are valid.

If $\Phi$ and $b_{-j}$ are chosen in the above way, then \eqref{p_eilambdaphi_formulatwo} implies that 
\[
P(e^{i\lambda \Phi} b) = e^{i\lambda \Phi}\left[ \lambda q_1(x)   + \sum_{j=0}^{N} \lambda^{-j} q_{-j}(x) + \lambda^{-N} Pb_{-N} \right]
\]
where each $q_j(x)$ vanishes to infinite order at $x_n = 0$ and is compactly supported in  $Q_{\eps}$. Thus, for any $k \geq 0$ there is $C_k > 0$ so that $\abs{q_j} \leq C_k x_n^k$ in $Q_{\eps}$, and consequently 
\[
\abs{P(e^{i\lambda \Phi} b)} \leq e^{-\lambda \im(\Phi)} \left[ \lambda C_k x_n^k + C \lambda^{-N} \right].
\]
Since $\im(\Phi) = x_n$ in $Q_{\eps}$ we have 
\begin{align*}
\norm{P(e^{i\lambda \Phi} b)}_{L^2(\Omega)}^2 &\leq C_k \int_{Q_{\eps}} e^{-2\lambda x_n} \left[ \lambda^2 x_n^{2k} + \lambda^{-2N} \right] \,dx \\
 &\leq C_k \int_{\abs{x'} < \eps} \int_0^{\infty} e^{-2 x_n} \left[ \lambda^{1-2k} x_n^{2k} + \lambda^{-1-2N} \right] \,dx_n \,dx'.
\end{align*}
Choosing $k=N+1$ and computing the integrals over $x_n$, we get that 
\[
\norm{P(e^{i\lambda \Phi} b)}_{L^2(\Omega)}^2 \leq C_N \lambda^{-2N-1}.
\]
It is also easy to compute that 
\[
\norm{e^{i\lambda \Phi} b}_{H^1(\Omega)} \sim \lambda^{1/2}.
\]
Thus, choosing $a = \lambda^{-1/2} b$, we have proved all the claims except \eqref{boundary_determination_limit}.

To show \eqref{boundary_determination_limit}, we observe that 
\[
\nabla v = e^{i\lambda \Phi} \left[ i\lambda (\nabla \Phi) a + \nabla a \right].
\]
Using a similar formula for $\tilde{v} = e^{i\lambda \Phi} \tilde{a}$ (where $\Phi$ is independent of the conductivity), we have 
\[
\mathrm{dist}(x,\p \Omega)^k f \nabla v \cdot \ol{\nabla \tilde{v}} = x_n^k f e^{-2\lambda x_n} \left[ \lambda^2 \abs{\nabla \Phi}^2 a \ol{\tilde{a}} + \lambda^1 [ \cdots ] + \lambda^0 [ \cdots ]  \right].
\]
Now $\abs{\nabla \Phi}^2 = 2$ and $a = \lambda^{-1/2} b$ where $\abs{b} \lesssim 1$, and similarly for $\tilde{a}$. Hence 
\begin{align*}
 &\lambda^k \int_{\Omega} \mathrm{dist}(x,\p \Omega)^k f \nabla v \cdot \ol{\nabla \tilde{v}} \,dx \\
 &= \lambda^{k+1} \int_{\mR^{n-1}} \int_0^{\infty} x_n^k e^{-2\lambda x_n} f \left[ 2 b \ol{\tilde{b}} + O_{L^{\infty}}(\lambda^{-1}) \right] \,dx_n \,dx'.
\end{align*}
We can change variables $x_n \to x_n/\lambda$ and use dominated convergence to take the limit as $\lambda \to \infty$. The limit is 
\[
c_k \int_{\mR^{n-1}} f(x',0) b(x',0) \ol{\tilde{b}(x',0)} \,dx' = c_k \int_{\mR^{n-1}} f(x',0) \abs{\chi(x')}^2 \,dx'
\]
where $c_k = 2 \int_0^{\infty} x_n^k e^{-2x_n} \,dx_n \neq 0$.

The proof is complete in the case when $x_0 = 0$ and $\p \Omega$ is flat near $0$. In the general case, we choose boundary normal coordinates $(x',x_n)$ so that $x_0$ corresponds to $0$ and $\Omega$ near $x_0$ locally corresponds to $\{ x_n > 0 \}$. The equation $\nabla \cdot (\gamma \nabla u) = 0$ in the new coordinates becomes an equation 
\[
\nabla \cdot (\gamma A \nabla u) = 0 \text{ in $\{ x_n > 0 \}$}
\]
where $A$ is a smooth positive matrix only depending on the geometry of $\Omega$ near $x_0$. The construction of $v$ now proceeds in a similary way as above, except that the equation \eqref{ptwo_phi_equation} for the phase function $\Phi$ can only be solved to infinite order on $\{ x_n = 0 \}$ instead of solving it globally in $\Omega$.
\end{proof}

\end{document}